\documentclass{article}

\usepackage{amsmath, amsthm, amssymb,amsbsy}
\usepackage{graphicx}
\usepackage[colorlinks,citecolor=blue]{hyperref}
\usepackage{tikz}
\usetikzlibrary{shapes,arrows,positioning,calc, decorations.pathmorphing, patterns}
\usepackage{bbm}
\usepackage{float}
\usepackage{enumitem}
\usepackage{comment}
\usepackage{stmaryrd}

\usepackage[left=3cm, right=4cm, bottom=3cm]{geometry}

\usepackage[caption=false,font=normalsize,labelfon
t=sf,textfont=sf]{subfig}

\title{Universal Lossless Compression of Graphical Data}
\author{Payam Delgosha and Venkat Anantharam\footnote{This paper was
    presented in part at 2017 IEEE International Symposium on Information Theory}\\[2mm]
\small Department of Electrical Engineering and Computer Sciences\\
\small University of California, Berkeley\\
\small \{pdelgosha, ananth\} @ berkeley.edu
}

\newcommand{\evwrt}[2]{\mathbb{E}_{#1} \left [ #2 \right ] }
\newcommand{\pr}[1]{\mathbb{P} \left ( #1 \right ) }
\newcommand{\norm}[1]{\left \Vert #1 \right \Vert}
\newcommand{\snorm}[1]{\Vert #1 \Vert}
\newcommand{\one}[1]{\mathbbm{1} \left [ #1 \right ]}
\newcommand{\oneu}[1]{\mathbbm{1}_{#1}}

\newtheorem{lem}{Lemma}
\newtheorem{thm}{Theorem}
\newtheorem{definition}{Definition}
\newtheorem{prop}{Proposition}

\newtheorem{rem}{Remark}

\newcommand{\mG}{\mathcal{G}}

\newcommand{\vm}{\vec{m}}
\newcommand{\vu}{\vec{u}}
\newcommand{\vd}{\vec{d}}

\newcommand{\vmn}{\vec{m}^{(n)}}
\newcommand{\vun}{\vec{u}^{(n)}}

\newcommand{\mn}{m^{(n)}}
\newcommand{\tmn}{\widetilde{m}^{(n)}}
\newcommand{\vtmn}{\vec{\widetilde{m}}^{(n)}}
\newcommand{\un}{u^{(n)}}

\newcommand{\tun}{\widetilde{u}^{(n)}}
\newcommand{\vtun}{\vec{\widetilde{u}}^{(n)}}
\newcommand{\tG}{\widetilde{G}}

\newcommand{\mA}{\mathcal{A}}
\newcommand{\mP}{\mathcal{P}}

\newcommand{\mGb}{\bar{\mathcal{G}}}
\newcommand{\mTb}{\bar{\mathcal{T}}}
\newcommand{\mGn}{\mathcal{G}^{(n)}}

\newcommand{\mGnmnun}{\mathcal{G}^{(n)}_{\vmn, \vun}}

\newcommand{\Gn}{G^{(n)}}

\newcommand{\tGn}{\widetilde{G}^{(n)}}

\newcommand{\reals}{\mathbb{R}}

\newcommand{\nats}{\mathbb{N}}

\newcommand{\ER}{Erd\H{o}s--R\'{e}nyi }
\newcommand{\LP}{L\'{e}vy--Prokhorov }

\newcommand{\dlp}{d_\text{LP}} 
\newcommand{\bch}{ \Sigma} 
\newcommand{\bchover}{\overline{\Sigma}}
\newcommand{\bchunder}{\underbar{$\Sigma$}}
\newcommand{\condmnun}{|_{(\vmn, \vun)}} 

\let\oldmarginpar\marginpar
\renewcommand{\marginpar}[2][rectangle,draw,rounded corners,text width = 3cm, scale=0.7]{%
        \oldmarginpar{%
          \tikz \node at (0,0) [#1]{#2};}%
        }

\newcommand{\edgemark}{\Xi} 
\newcommand{\vermark}{\Theta} 

\newcommand{\vtype}{\Pi} 
\newcommand{\vvtype}{\vec{\Pi}} 
\newcommand{\vdeg}{\vec{\deg}} 

\newcommand{\nat}{\mathsf{nats}}
\allowdisplaybreaks 

\definecolor{bluenodecolor}{RGB}{62,126,176}
\definecolor{rednodecolor}{RGB}{173,61,58}

\definecolor{blueedgecolor}{RGB}{3,151,255}
\definecolor{orangeedgecolor}{RGB}{255,149,8}

\tikzstyle{nodeB} = [fill=bluenodecolor, circle, inner sep = 1.7pt]
\tikzstyle{nodeR} = [fill=rednodecolor, rectangle, inner sep = 2.3pt]
\tikzstyle{edgeB} = [very thick, blueedgecolor]
\tikzstyle{edgeO} = [very thick, orangeedgecolor, decoration = {zigzag,segment length = 0.2cm, amplitude = 0.5mm},decorate]

\newcommand{\drawedge}[4]{\draw[edge#3] (#1) -- ($(#1)!0.5!(#2)$); \draw[edge#4]
  (#2) -- ($(#2)!0.5!(#1)$);}
\newcommand{\nodelabel}[3]{\node at ($(#1)+(#2:5mm)$) {#3};}
\newcommand{\nodelabeldist}[4]{\node at ($(#1)+(#2:#3)$) {#4};}

\begin{document}

\colorlet{Cyan}{cyan}
\colorlet{Orange}{orange}
\tikzstyle{Node} = [circle,fill,inner sep=1.5pt]
\tikzstyle{Node2} = [rectangle,fill,inner sep=2pt]
\tikzstyle{Root} = [circle,fill=magenta,inner sep=1.7pt]
\tikzstyle{Root2} = [rectangle,fill=magenta,inner sep=2.5pt]
\tikzstyle{Cedge} = [Cyan, thick]
\tikzstyle{Oedge} = [Orange, densely dotted, very thick]

\maketitle

\begin{abstract}
Graphical data is comprised of a graph with marks on its edges
and vertices. The mark indicates the value of some attribute associated to 
the respective edge or vertex.
Examples of such data arise in social networks, molecular and systems biology,
and web graphs, as well as in several other application areas.
Our goal is to design schemes that can efficiently compress such graphical data without making  assumptions about its stochastic properties. Namely,
we wish
to develop a universal compression algorithm for graphical data sources. To formalize this goal, we employ the framework of local weak convergence, also called the objective method, which provides a technique to think of a marked graph as a kind of stationary stochastic processes, stationary with respect to movement between vertices of the graph. 
In recent work, we have generalized
a notion of entropy for unmarked graphs in this framework, due to Bordenave and Caputo, to the case of marked graphs.
We use this notion to evaluate the efficiency of a compression scheme. 
The lossless compression scheme we propose in this paper is then proved to be universally optimal in a precise technical sense. It is also capable of performing local data queries in the compressed form. 
\end{abstract}

\section{Introduction}
\label{sec:introduction}

Modern data often arrives in a form that is indexed by graphs or other combinatorial structures. 
This a much richer class of data objects than the traditionally
studied time series or multidimensional time series. 
Examples of such {\em graphical data} arise, for instance, in social networks, molecular and systems biology, and web graphs, as well as in several other application areas.\ 
An instance of graphical data arising in a social network 
would be a snapshot view of the network at a given time. Here 
the graph might describe whether a pair of individuals has ever
had an interaction,
while the marks on the vertices represent some characteristics of the
individuals currently of interest for the data analysis task, e.g. their
preference for coffee versus tea, and the marks on the edges the characteristics of their interaction, e.g. whether they are friends
or not.
Often the graph underlying the data is large.
Designing efficient compression schemes to store and analyze the data is 
therefore of significant importance. 
It is also desirable that the efficiency guarantee of the compression 
scheme not be dependent on the presumed accuracy of a stochastic model for it.
This motivates the desire for a universal compression scheme for graphical data sources.
Moreover, since data analysis is often the ultimate goal, the ability to make a broad class of data queries in the compressed form would be valuable.
This paper develops techniques to address these goals and achieves them
in a specific technical sense.

\subsection{Prior work}
\label{sec:prior-work}

The literature on compression and on evaluating the information content of graphical data can be divided into two categories based on whether there is a stochastic model for the graphical data or not. 
To the best of our knowledge, none of the existing works address efficient universal data compression of such data. Works that do not consider a stochastic model for the data generally propose a scheme and illustrate its performance through some analysis and simulation. For example, Boldi and Vigna proposed the WebGraph framework to encode the web graph, where each node represents a URL, and two nodes are connected if there is link between them \cite{boldi2004webgraph}. 
Later, Boldi et al.\ proposed a method called layered label propagation as a  compression scheme for social networks \cite{boldi2011layered}. 

Among models making stochastic assumptions, Choi and Szpankowski studied structural compression of the \ER ensemble $\mG(n,p)$ 
\cite{choi2012compression}. 
Magner et~al.\ presented a compression scheme for binary trees with correlated
vertex names \cite{magner2016lossless}.
Basu and Varshney addressed the problem of source coding for deep neural
networks \cite{basu2017universal}.
Aldous and Ross studied models of sparse random graphs 
(i.e. the number of edges is of the order of the number of vertices),
with vertex labels  \cite{aldous2014entropy}. They considered several models on
sparse random graphs, and studied the asymptotic behavior of the entropy of such
models. They observed that the leading term in these models scales as $n \log
n$, where $n$ is the number of vertices in the graph.

\subsection{Our contributions}
\label{sec:our-contributions}

We assume that a graph is presented to us where each vertex carries a mark coming from some fixed finite set of possible vertex marks
and each edge carries two marks, one towards each of its endpoints,  coming from some fixed finite set of 
possible edge marks.
Our goal is to design a lossless compression scheme capable of compressing this marked graph in order
to store it in an efficient way. Furthermore, we would like to do this in a universally efficient fashion, meaning that the compressor does not make any prior assumption about the statistical properties of the data, but is nevertheless efficient.

In order to make sense of the question, we employ the framework of local weak convergence,
also called the objective method
 \cite{BenjaminiSchramm01rec,aldous2004objective,aldous2007processes}. 
 This framework views a finite marked graph as a probability distribution
 on a space of rooted marked graphs. 
 Convergence of a sequence of finite marked graphs is then defined as the
 convergence of this sequence of probability 
 distributions on rooted marked graphs. Both the prelimit and the limit
 can then be thought of as a kind of rooted marked graph valued 
 stationary stochastic process, stationary with respect to changes in the root.
 
In the case of unmarked graphs, a notion of entropy, which applies for such
limits and, crucially, works on a per vertex basis, is defined in
\cite{bordenave2014large}. We have generalized this entropy notion so that it works for the limits
 arising from sequences of marked graphs  \cite{delgosha2019notion}.
Efficient universal compression is then formalized as the ability to 
compress marked graphs such that one asymptotically
pays the minimal entropy cost, the asymptotics being in the number of vertices of the graph, without prior knowledge of the limit. 
Our main contribution  is to design such a universal compression scheme and prove its optimality
in  a precise technical sense.

\begin{figure}
  \centering
  \begin{tikzpicture}[scale=0.7]
    \begin{scope}[yshift=0.7cm]
    \node[nodeB] (n1) at (0,2) {};
    \node[nodeB] (n2) at (-1,1) {};
    \node[nodeB] (n3) at (1,1) {};
    \node[nodeR] (n4) at (0,0) {};
  \end{scope}
  \begin{scope}[yshift=-0.7cm]
    \node[nodeR] (n5) at (0,0) {};
    \node[nodeB] (n6) at (-1,-1) {};
    \node[nodeB] (n7) at (1,-1) {};
    \node[nodeB] (n8) at (0,-2) {};
  \end{scope}
  \nodelabel{n1}{90}{1};
  \nodelabel{n2}{180}{2};
  \nodelabel{n3}{0}{3};
  \nodelabel{n4}{180}{4};
  \nodelabel{n5}{180}{5};
  \nodelabel{n6}{180}{6};
  \nodelabel{n7}{0}{7};
  \nodelabel{n8}{270}{8};

  \drawedge{n1}{n2}{B}{O}
  \drawedge{n1}{n3}{B}{O}
  \drawedge{n8}{n6}{B}{O}
  \drawedge{n8}{n7}{B}{O}

  \drawedge{n2}{n4}{B}{B}
  \drawedge{n3}{n4}{B}{B}
  \drawedge{n5}{n6}{B}{B}
  \drawedge{n5}{n7}{B}{B}
  \draw[edgeO] (n4) -- (n5);
  \end{tikzpicture}
  

\caption[A marked graph]{\label{fig:marked-graph} A marked graph $G$ on the vertex set
  $\{1,\dots,8\}$ where edges carry marks from $\edgemark =
  \{\text{\color{blueedgecolor} Blue (solid)}, \text{\color{orangeedgecolor}
    Orange (wavy)} \}$ (e.g.\ $\xi_G(1,2)= \text{\color{orangeedgecolor} Orange}$
  while $\xi_G(2,1) = \text{\color{blueedgecolor} Blue}$; also, $\xi_G(2,4) =
  \xi_G(4,2) = \text{\color{blueedgecolor} Blue}$) and vertices carry marks from
$\vermark = \{\tikz{\node[nodeB] at (0,0) {};}, \tikz{\node[nodeR] at (0,0)
  {};}\}$ (e.g. $\tau_G(3) = \tikz{\node[nodeB] at (0,0) {};}$).
}
\end{figure}
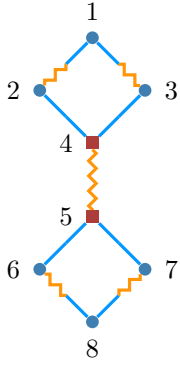

The paper is organized as follows. 
In Section~\ref{sec:preliminaries}, we describe the language of local weak convergence.
Section~\ref{sec:noti-entr-proc} discusses our notion of entropy. In Section~\ref{sec:main-results}, we precisely formulate the problem and state the main results. In Section~\ref{sec:univ-coding-algor}, we introduce our universal compression scheme and sketch the proof of its optimality. We make some concluding remarks in Section~\ref{sec:conclusion}.

\subsection{Notational conventions}
\label{sec:notational}

All the logarithms in this paper are to the natural base unless otherwise
stated. 
We therefore use nats instead of
bits as the unit of information. 
For two sequences $(a_n, n \ge 1)$ and $(b_n, n \ge 1)$ 
of positive real numbers, we write $a_n = O(b_n)$ if 
$\sup_n  a_n / b_n < \infty$,
and we write $a_n = o(b_n)$ if 
$a_n / b_n \rightarrow 0$ as $n \to \infty$. 
We write $\{0,1\}^* - \emptyset$ 
for the set of sequences of zeros and ones of
finite length, excluding the empty sequence.
For $x \in \{0,1\}^* - \emptyset$, 
we denote the length of the sequence
$x$ by
$\nat(x)$, which is obtained by
multiplying the length of $x$ in bits by $\log 2$. 
$\mathbb{Z}$ denotes the set of integers.
We write $:=$ and $=:$ for equality by definition. 
Other notation used in this document is introduced at its first appearance.


\section{Preliminaries}
\label{sec:preliminaries}

Let $G$ be a simple graph, i.e.\ one with no self loops or multiple edges.
We consider graphs which may have either a finite or a
countably infinite number of vertices. 
Let $\vermark$ and $\edgemark$ be finite sets called the 
set of {\em vertex marks} and the set of {\em edge marks} respectively.
A (simple) {\em marked graph} is a
simple graph $G$ where vertices carry marks from the vertex mark set
$\vermark$, and each edge carries two marks from the edge mark set $\edgemark$,
one towards each of its endpoints. We denote the mark of a vertex $v$ in a
marked graph $G$ by $\tau_G(v)$, and the mark of an edge $(i,j)$ in $G$ towards
vertex $j$ by $\xi_G(i,j)$. For a marked or an unmarked graph $G$, let $V(G)$
denote the set of vertices in $G$. For a marked or an unmarked graph $G$ and
vertices $i,j \in V(G)$, we
write $i \sim_G j$ to denote that $i$ and $j$ are adjacent in $G$. 
All graphs in this paper, marked or unmarked, are assumed to be simple. Hence,
we may drop the term ``simple'' when referring to graphs. 
See Figure~\ref{fig:marked-graph} for an example. 
A {\em marked tree} is a marked graph $T$  where the underlying graph  is a tree.

Two marked graphs $G$ and $G'$ 
are said to be isomorphic, and we write $G \equiv G'$, if there is a bijection $\phi: V(G) \rightarrow V(G')$ such that:
\begin{enumerate}
\item $\tau_{G'}(\phi(i)) = \tau_G(i)$ for all $i \in V(G)$;
\item $i \sim_G j$ iff $\phi(i) \sim_{G'} \phi(j)$;
\item For $i \sim_G j$, we have $\xi_{G'}(\phi(i), \phi(j)) = \xi_G(i,j)$. 
\end{enumerate}

To better understand this notion, let $\mathcal{S}_n$ denote the permutation
group on the set $\{1, \dots, n\}$. For a permutation $\pi \in \mathcal{S}_n$
and a marked graph $G$ on the vertex set $\{1, \dots, n\}$, let $\pi G$ be the
marked graph on the same vertex set after the permutation $\pi$ is applied on
the vertices. Namely, vertex $\pi(i)$ has mark $\tau_G(i)$ for $1 \leq i\leq n$,
and for each edge $(i,j)$ in $G$, we place an edge between the vertices $\pi(i)$ and $\pi(j)$ 
in $\pi G$, with mark $\xi_G(i,j)$ towards $\pi(j)$, and mark $\xi_G(j,i)$
towards $\pi(i)$.
Then each $\pi G$ is isomorphic to $G$ and every marked graph
that is isomorphic to $G$ is of the form $\pi G$ for some
$\pi \in \mathcal{S}_n$. 
See Figure~\ref{fig:graph-permutation} for an example.



\begin{figure}
  \centering
  \subfloat[]{
    \begin{tikzpicture}[scale=0.8]
      \node[nodeB] (n1) at (0,1) {};
      \node[nodeR] (n2) at (-1,0) {};
      \node[nodeR] (n3) at (1,0) {};
      \node[nodeB] (n4) at (0,-1) {};
      \nodelabel{n1}{90}{1};
      \nodelabel{n2}{180}{2};
      \nodelabel{n3}{0}{3};
      \nodelabel{n4}{270}{4};
      \drawedge{n1}{n2}{B}{O}
      \drawedge{n1}{n3}{B}{B}
      \drawedge{n2}{n4}{B}{B}
      \drawedge{n3}{n4}{O}{B}
    \end{tikzpicture}
  }
  \hfill
  \subfloat[]{
    \begin{tikzpicture}[scale=0.8]
            \node[nodeB] (n1) at (0,1) {};
      \node[nodeR] (n2) at (-1,0) {};
      \node[nodeB] (n3) at (1,0) {};
      \node[nodeR] (n4) at (0,-1) {};
      \nodelabel{n1}{90}{1};
      \nodelabel{n2}{180}{2};
      \nodelabel{n3}{0}{3};
      \nodelabel{n4}{270}{4};
      \drawedge{n1}{n2}{B}{O}
      \drawedge{n1}{n4}{B}{B}
      \drawedge{n2}{n3}{B}{B}
      \drawedge{n3}{n4}{B}{O}
    \end{tikzpicture}
  }
  \hfill
  \subfloat[]{
    \begin{tikzpicture}[scale=0.8]
            \node[nodeB] (n1) at (0,1) {};
      \node[nodeR] (n2) at (-1,0) {};
      \node[nodeR] (n3) at (1,0) {};
      \node[nodeB] (n4) at (0,-1) {};
      \nodelabel{n1}{90}{1};
      \nodelabel{n2}{180}{2};
      \nodelabel{n3}{0}{3};
      \nodelabel{n4}{270}{4};
      \drawedge{n1}{n2}{B}{O}
      \drawedge{n1}{n3}{B}{B}
      \drawedge{n2}{n4}{B}{B}
      \drawedge{n3}{n4}{O}{B}
    \end{tikzpicture}
  }
  \caption[Graph permutation]{\label{fig:graph-permutation} (a) A marked graph $G$ on the vertex set $\{1,\dots,4\}$ with $\edgemark =
  \{\text{\color{blueedgecolor} Blue (solid)}, \text{\color{orangeedgecolor}
    Orange (wavy)} \}$ and
$\vermark = \{\tikz{\node[nodeB] at (0,0) {};}, \tikz{\node[nodeR] at (0,0)
  {};}\}$; (b) $\pi_1 G$ where $\pi_1 = (3\,\, 4)$; (c) $\pi_2 G$ where $\pi_2 = (1\,\, 4)(2\,\, 3)$. Note that $\pi_2 G = G$.
} 
\end{figure}
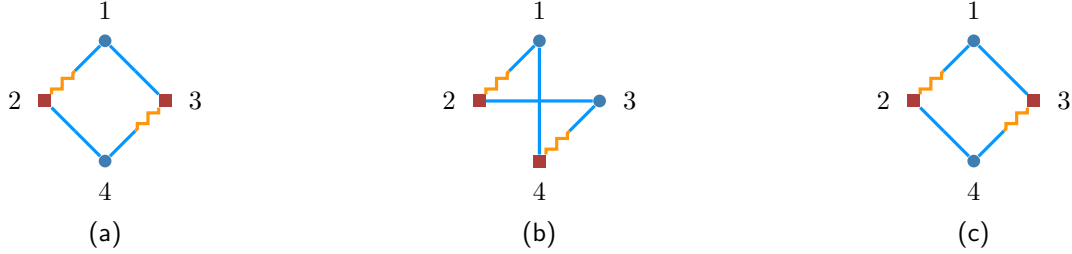









A {\em path} between two vertices $v$ and $w$ in 
the marked graph $G$, is a sequence of distinct vertices $v_0, v_1, \dots, v_k$,
such that $v_0 = v$, $v_k = w$ and, for all $1 \leq i \leq k$, we have $v_{i-1}
\sim_G v_i$.  The length of such a path is defined to be $k$. The {\em distance}
between $v$ and $w$ is the length of the shortest path connecting $v$ and $w$.
If there is no such path, the distance is defined to be $\infty$. 

Given a marked graph $G$, 
and a subset $S$ of its vertices, the subgraph induced by $S$ is 
the marked graph comprised of the vertices in $S$ 
and those edges in $G$
that have both their endpoints in $S$, 
with the vertex and edge marks being inherited from $G$.
The {\em connected component} of a vertex $v \in V(G)$ is the subgraph of $G$ induced by 
the vertices that are at a finite distance from $v$. 
We write $G_v$ for the connected component of $v \in V(G)$. Note that $G_v$ is
a connected marked graph.

For 
a marked graph 
$G$
and a vertex $v \in V(G)$, we denote the degree of $v$,
i.e.\ the  
number of edges connected to $v$,
by $\deg_{G}(v)$.
Given $x, x' \in \edgemark$, we let $\deg^{x,x'}_{G}(v)$ denote the number of
neighbors $w \sim_G v$ such that $\xi_{G}(w,v) = x$ and $\xi_G(v,w) = x'$.
A marked graph $G$ is called {\em locally finite} if the degree of every
vertex in the graph is finite.

The focus on how a marked graph looks from the point of view of each of its
vertices is the key conceptual ingredient in the theory of local weak 
convergence.
For this, we introduce the notion of a
{\em rooted} marked
graph and the notion of isomorphism of rooted marked graphs. 
Roughly speaking, a rooted marked graph should be thought of as
a marked graph as seen from a specific vertex in it and the
notion of two rooted marked
graphs being isomorphic as capturing the idea that the respective marked
graphs as seen from the respective distinguished vertices look the same. 
Notice that it is natural to restrict attention to the connected component
containing the root when making such a definition, because,  
roughly speaking, a vertex of the marked graph should only be able to see the component to which it belongs.


For a precise definition, consider a marked graph $G$ and a distinguished vertex $o \in V(G)$.
The pair $(G,o)$ is called a rooted marked graph. 
We call two rooted marked graphs $(G, o)$ and $(G',o')$ isomorphic and write $(G, o) \equiv (G',o')$ if $G_o \equiv G'_{o'}$ through a bijection $\phi: V(G_o) \rightarrow V(G'_{o'})$ preserving the root, i.e.\ $\phi(o) = o'$.
This notion of isomorphism defines an equivalence relation 
on rooted marked graphs. Note that in order to determine
if two rooted marked graphs are isomorphic (as rooted marked graphs)
it is only necessary to 
examine the connected component of the root in each of the marked graphs.
Let $[G, o]$ denote the equivalence class corresponding to $(G_o, o)$. 
In the sequel, we will only use this notion for locally finite graphs.

For a rooted marked graph $(G, o)$ and integer $h\geq 1$, let $(G, o)_h$ be
the subgraph of  $G$ rooted at $o$ induced by  vertices with distance no more
than $h$ from $o$. If $h=0$,  $(G, o)_h$ is defined to  be the isolated root $o$ with mark $\tau_{G}(o)$. 
Moreover, let $[G, o]_h$ be the  equivalence class corresponding to $(G, o)_h$, i.e.\ $[G, o]_h := [(G, o)_h]$.
Note that $[G, o]_h$ depends only on  $[G, o]$.

Finally, note that all the notions introduced in this subsection have
the obvious parallels for unmarked graphs. One could simply walk through 
the definitions while taking each of the mark sets $\vermark$ and 
$\edgemark$ to be of cardinality $1$.



To close this section we introduce some notation that will only be 
needed when we develop our compression algorithm for graphical data.
For a locally finite graph $G$ and integer $\Delta$, let $G^\Delta$ be the graph with the same vertex set that includes only those
edges of $G$ such that the degrees of both their endpoints are 
at most $\Delta$ 
(without reference to their marks). 
Another way to put this is that to arrive at $G^\Delta$ from $G$
we remove all the edges in $G$ 
that are 
connected to vertices 
with degree strictly bigger than $\Delta$.
This construction is used as a technical device in the proof of the main 
result, the main point being that 
the maximum degree in $G^\Delta$ 
is at most $\Delta$.


\subsection{The Framework of Local Weak Convergence}
\label{sec:local-weak-conv}

In this section, we review the framework of local weak convergence
of graphs, also called the objective method, with our focus being on 
marked graphs.
See \cite{benjamini2011recurrence, aldous2004objective, aldous2007processes} for more details. Throughout our discussion, we assume that the set of edge marks, $\edgemark$, and the set of vertex marks, $\vermark$, are finite and fixed. 

Let $\mGb_*$ be the space of equivalence classes 
$[G, o]$
arising from locally finite rooted marked graphs $(G,o)$.
We recall again that in defining $[G, o]$ all that matters about
$(G,o)$ is the connected component of the root.
We define the metric $\bar{d}_*$ on $\mGb_*$ as follows: given $[G, o]$ and $[G',o']$, let $\hat{h}$ be the supremum over all integers $h\geq 0$ such that $(G, o)_h \equiv (G', o')_h$, where 
$(G, o)$ and $(G',o')$ are arbitrary members in equivalence classes $[G, o]$ and $[G', o']$ respectively\footnote{As all elements in an equivalence class are isomorphic, the definition is invariant under the choice of the representatives.
}.
If there is no such $h$ (which can only happen if $\tau_{G}(o) \neq \tau_{G'}(o')$), we define $\hat{h} = 0$. 
With this, $\bar{d}_*([G,o], [G',o'])$ is defined to be $1/(1+\hat{h})$. One can check that $\bar{d}_*$ is a metric; in particular, it satisfies the triangle inequality.
Moreover, 
$\mGb_*$ together with this metric is a Polish space, i.e. a 
complete separable metric space \cite{aldous2007processes}.\footnote{It is also possible to define a metric on the space of equivalence
classes of locally finite rooted marked graphs when the mark space is 
itself an arbitrary Polish space, by modifying the definition of the 
distance between two rooted marked graphs to penalize differences
in marks according to how close they are to each other. However, this
more general definition is not needed for the purposes of this paper.}
Let $\mTb_*$ be the subset of $\mGb_*$ comprised of the 
equivalence classes $[G, o]$ arising from some $(G,o)$ where
the graph underlying $G$ is a tree. 
In the sequel we will 
think of $\mGb_*$ as a Polish space with the metric $\bar{d}_*$
defined above, rather than just a set. Note that $\mTb_*$ is a closed subset of $\mGb_*$.

For a Polish space 
$\Omega$, let $\mP(\Omega)$ denote the set of Borel probability measures on
$\Omega$. We say that a sequence of measures $\mu_n$ on $\Omega$ converges
weakly to $\mu \in \mP(\Omega)$ and write $\mu_n \Rightarrow \mu$, if for any
bounded continuous function on $\Omega$, we have $\int f d \mu_n \rightarrow
\int f d \mu$.
It can be shown that it suffices to verify this condition only for
uniformly continuous and bounded functions \cite{billingsley2013convergence}.
For a Borel set $B \subset \Omega$, the $\epsilon$--extension of
$B$, denoted by $B^\epsilon$, is defined as the union of the open balls with
radius $\epsilon$ centered around the points in $B$. For two probability
measures $\mu$ and $\nu$ in $\mP(\Omega)$, the \LP distance $\dlp(\mu, \nu)$ is
defined to be the infimum of all $\epsilon>0$ such that for all Borel sets $B
\subset \Omega$ we have $\mu(B) \leq \nu(B^\epsilon) + \epsilon$ and  $\nu(B) \leq \mu(B^\epsilon) + \epsilon$.
It is known that the \LP distance metrizes the topology of weak convergence 
on the space of probability distributions on a Polish space
(see, for instance, \cite{billingsley2013convergence}).
For $x \in \Omega$, let $\delta_x$ be the Dirac measure at $x$.

For a finite marked graph $G$, define $U(G) \in \mP(\mGb_*)$ as 
\begin{equation}
  \label{eq:UG}
  U(G) := \frac{1}{|V(G)|} \sum_{o \in V(G)} \delta_{[G, o]}.
\end{equation}
Note that $U(G) \in \mP(\mGb_*)$. In creating $U(G)$
from $G$, we have created a probability distribution on 
rooted marked graphs from the given marked graph $G$ 
by rooting the graph at a vertex chosen uniformly at random. 
Furthermore, for an integer $h \geq 1$, let 
\begin{equation}
\label{eq:UkG}
  U_h(G) := \frac{1}{|V(G)|} \sum_{o \in V(G)} \delta_{[G, o]_h}.
\end{equation}
We then have $U_h(G) \in \mP(\mGb_*)$.
See Figure~\ref{fig:UG} for an example.

We say that a probability distribution $\mu$ on $\mGb_*$ is the {\em local weak limit} of a sequence of finite marked graphs $\{G_n\}_{n=1}^\infty$ when $U(G_n)$ converges weakly to $\mu$ 
(with respect to the topology induced by the metric $\bar{d}_*$).
This turns out to be equivalent to the condition that, for any finite depth $h \ge 0$, the structure of $G_n$ from the point of view of a root chosen uniformly at random and then looking around it only to depth $h$, converges in distribution to $\mu$ truncated up to depth $h$. This description of what is being captured by the definition justifies the term ``local weak convergence''.

\begin{figure}
  \centering
\hfill
\subfloat[]{
\centering
\begin{tikzpicture}[scale=0.9]
  \begin{scope}[yshift=2cm,xshift=1.5cm, scale=0.23]
    \node[nodeR] (n1) at (0,0) {};
    \node[nodeB] (n2) at (-4,-2) {};
    \node[nodeB] (n3) at (0,-2) {};
    \node[nodeR] (n4) at (4,-2) {};
    \node[nodeB] (n5) at (-2,-4) {};
    \node[nodeB] (n6) at (2,-4) {};
    \node[nodeB] (n7) at (6,-4) {};

    \drawedge{n1}{n2}{B}{B}
    \drawedge{n1}{n3}{B}{B}
    \draw[edgeO] (n1) -- (n4);
    \drawedge{n2}{n5}{O}{B}
    \drawedge{n3}{n5}{O}{B}
    \drawedge{n4}{n6}{B}{B}
    \drawedge{n4}{n7}{B}{B}

    \node at (0,2) {$\frac{1}{4}$};

  \end{scope}

  \begin{scope}[yshift=2cm,xshift=-1.5cm,scale=0.23]
    \node[nodeB] (n1) at (0,0) {};
    \node[nodeB] (n2) at (-2,-2) {};
    \node[nodeR] (n3) at (2,-2) {};
    \node[nodeB] (n4) at (0,-4) {};
    \node[nodeR] (n5) at (4,-4) {};
    \drawedge{n1}{n2}{O}{B}
    \drawedge{n1}{n3}{B}{B}
    \drawedge{n2}{n4}{B}{O}
    \drawedge{n3}{n4}{B}{B}
    \draw[edgeO] (n3) -- (n5);
    
    \node at (0,2) {$\frac{1}{2}$};

  \end{scope}

  \begin{scope}[scale=0.23]
    \node[nodeB] (n1) at (0,0) {};
    \node[nodeB] (n2) at (-2,-2) {};
    \node[nodeB] (n3) at (2,-2) {};
    \node[nodeR] (n4) at (0,-4) {};
    
    \drawedge{n1}{n2}{B}{O}
    \drawedge{n1}{n3}{B}{O}
    \drawedge{n2}{n4}{B}{B}
    \drawedge{n3}{n4}{B}{B}
    
    \node at (0,2) {$\frac{1}{4}$};
  \end{scope}
\end{tikzpicture}
\label{fig:U-d2}
}
\hfill\hfill\hfill
\subfloat[]{
\centering
\begin{tikzpicture}[scale=0.9]
  \begin{scope}[yshift=1cm,xshift=2.5cm,scale=0.23]
        \node[nodeR] (n1) at (0,0) {};
    \node[nodeB] (n2) at (-4,-2) {};
    \node[nodeB] (n3) at (0,-2) {};
    \node[nodeR] (n4) at (4,-2) {};
    \node[nodeB] (n5) at (-2,-4) {};
    \node[nodeB] (n6) at (2,-4) {};
    \node[nodeB] (n7) at (6,-4) {};
    \node[nodeB] (n8) at (4,-6) {};
    
    \drawedge{n1}{n2}{B}{B}
    \drawedge{n1}{n3}{B}{B}
    \draw[edgeO] (n1) -- (n4);
    \drawedge{n2}{n5}{O}{B}
    \drawedge{n3}{n5}{O}{B}
    \drawedge{n4}{n6}{B}{B}
    \drawedge{n4}{n7}{B}{B}
    \drawedge{n6}{n8}{O}{B}
    \drawedge{n7}{n8}{O}{B}
    
    \node at (0,2) {$\frac{1}{4}$};

  \end{scope}

  \begin{scope}[yshift=1cm,xshift=-3cm,scale=0.23]

    \node[nodeB] (n1) at (0,0) {};
    \node[nodeB] (n2) at (-2,-2) {};
    \node[nodeR] (n3) at (2,-2) {};
    \node[nodeB] (n4) at (0,-4) {};
    \node[nodeR] (n5) at (4,-4) {};
    \node[nodeB] (n6) at (2,-6) {};
    \node[nodeB] (n7) at (6,-6) {};
    \node[nodeB] (n8) at (4,-8) {};
    
    \drawedge{n1}{n2}{O}{B}
    \drawedge{n1}{n3}{B}{B}
    \drawedge{n2}{n4}{B}{O}
    \drawedge{n3}{n4}{B}{B}
    \drawedge{n5}{n6}{B}{B}
    \drawedge{n5}{n7}{B}{B}
    \drawedge{n6}{n8}{O}{B}
    \drawedge{n7}{n8}{O}{B}
    \draw[edgeO] (n3) -- (n5);

    \node at (0,2) {$\frac{1}{2}$};

  \end{scope}

  \begin{scope}[scale=0.4,yshift=-3cm]
        \begin{scope}[yshift=0.7cm]
    \node[nodeB] (n1) at (0,2) {};
    \node[nodeB] (n2) at (-1,1) {};
    \node[nodeB] (n3) at (1,1) {};
    \node[nodeR] (n4) at (0,0) {};
  \end{scope}
  \begin{scope}[yshift=-0.7cm]
    \node[nodeR] (n5) at (0,0) {};
    \node[nodeB] (n6) at (-1,-1) {};
    \node[nodeB] (n7) at (1,-1) {};
    \node[nodeB] (n8) at (0,-2) {};
  \end{scope}

  \drawedge{n1}{n2}{B}{O}
  \drawedge{n1}{n3}{B}{O}
  \drawedge{n8}{n6}{B}{O}
  \drawedge{n8}{n7}{B}{O}

  \drawedge{n2}{n4}{B}{B}
  \drawedge{n3}{n4}{B}{B}
  \drawedge{n5}{n6}{B}{B}
  \drawedge{n5}{n7}{B}{B}
  \draw[edgeO] (n4) -- (n5);

    \node at (0,4) {$\frac{1}{4}$};
  \end{scope}
\end{tikzpicture}
\label{fig:U-U}
}
\hfill
\caption{\label{fig:UG} 
With $G$ being the graph from Figure~\ref{fig:marked-graph}, 
(a) illustrates $U_2(G)$, which is a probability distribution on rooted marked graphs of depth at most 2 and
(b) depicts $U(G)$, which is a probability distribution on $\mGb_*$. 
}
\end{figure}
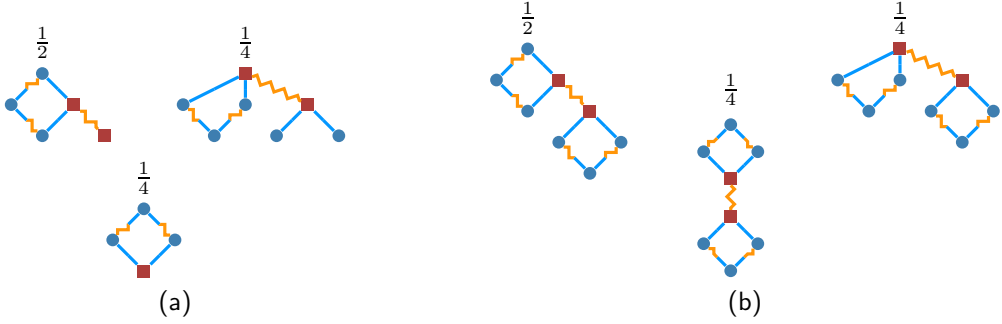

In fact, $U_h(G)$ could be thought of as the ``depth $h$ empirical distribution'' of the marked graph $G$. On the other hand, a probability distribution $\mu \in \mP(\mGb_*)$ that arises as a local weak limit plays the role of a stochastic process on graphical data, and a sequence of marked graphs $\{G_n\}_{n=1}^\infty$ could be thought of as being 
asymptotically distributed like
this process when $\mu$ is the local weak limit of the sequence.

The degree of a probability measure $\mu \in \mP(\mGb_*)$,
denoted by $\deg(\mu)$, is defined as
\begin{equation*}
  \deg(\mu) := \evwrt{\mu}{ \deg_{G}(o)} = \int \deg_{G}(o) d \mu([G, o]),
\end{equation*}
which is the expected degree of the root.
Similarly, for $\mu \in \mP(\mGb_*)$ and $x, x' \in \Xi$, let $\deg_{x,x'}(\mu)$ be defined as 
\begin{equation*}
  \deg_{x,x'}(\mu) := \int \deg_{G}^{x,x'}(o) d \mu([G, o]),
\end{equation*}
which is the expected number of edges  connected to the root with mark $x$
towards the root and mark $x'$ towards the other endpoint. 
We use the notation $\vdeg(\mu) := (\deg_{x,x'}(\mu): x, x' \in \edgemark)$.
Moreover, for $\mu \in \mP(\mGb_*)$ and $\theta \in \vermark$, let
\begin{equation}
\label{eq:vertype-probability}
  \vtype_\theta(\mu) := \mu(\{[G, o] \in \mGb_*: \tau_{G}(o) = \theta \}),
\end{equation}
which is  the probability under $\mu$ that the mark of the root is $\theta$. 
We use the notation $\vvtype(\mu) := (\vtype_\theta(\mu): \theta \in \vermark)$. 

All the preceding definitions and concepts have the obvious parallels in the 
case of unmarked graphs. These can be arrived at by simply walking through
the definitions while restricting the mark sets $\vermark$ and $\edgemark$
to be of cardinality $1$. It is convenient, however, to sometimes use
the special
notation for the unmarked case
that matches the one currently in use in the literature. 
We will therefore write $\mG_*$ for the set of rooted isomorphism classes of unmarked graphs. This is just the set $\mGb_*$ in the case where 
both $\edgemark$ and $\vermark$ are sets of cardinality $1$. 
We will also denote the metric on $\mG_*$ by $d_*$, which is
just $\bar{d}_*$ when both $\edgemark$ and $\vermark$ are sets of cardinality $1$. 


We next present some examples to illustrate the concepts defined so far.
\begin{enumerate}
\item Let $G_n$ be the finite lattice $\{-n, \dots n \} \times \{-n, \dots, n
  \}$ in $\mathbb{Z}^2$. As $n$ goes to infinity,  the local weak limit of this sequence is the distribution that gives probability one to the lattice $\mathbb{Z}^2$ rooted at the origin. The reason is that, if we fix a depth $h \ge 0$, then for $n$ large almost all of the vertices in $G_n$ cannot see the borders of the lattice when they look at the graph around them up to depth $h$, so
these vertices cannot locally distinguish the graph on which they live from the infinite lattice $\mathbb{Z}^2$.
\item Suppose $G_n$ is a cycle of length $n$. The local weak limit 
of this sequence of graphs gives probability one to an infinite $2$--regular tree
rooted at one of its vertices. The intuitive explanation for this is 
essentially identical to that for the preceding example.

\item Let $G_n$ be a realization of the sparse \ER graph $\mG(n, \alpha / n)$
where $\alpha > 0$, i.e.\  $G_n$ has $n$ vertices and each edge is independently present with
  probability $\alpha / n$. One can show that if all the $G_n$ are defined on
  a common probability space then, almost surely, the local weak limit of the
  sequence is the Poisson Galton--Watson tree with mean  $\alpha$, 
rooted at the initial vertex. To justify why this should be true without going
through the details, note that the degree of a vertex in $G_n$ is the sum of
$n-1$ independent Bernoulli random variables, each with parameter $\alpha /n$.
For $n$ large, this  approximately has a Poisson distribution with mean
$\alpha$. This argument could be repeated for any of the vertices to which the
chosen vertex is connected, which play the role of the offspring of the initial
vertex in the limit. The essential point is that the probability of having loops
in the neighborhood of a typical vertex up to a depth $h$ is  negligible
whenever $h$ is fixed and $n$ goes to infinity. 

\item \label{item:bipartite-example} Let $G_n$ be the marked bipartite graph on the 
$2n$ vertices $\{1, \ldots, 2n\}$,
the edge mark set having cardinality $1$ and the vertex mark set being
$\vermark = \{R,B\}$. Suppose $\{1, \ldots, n\}$ is the set of left
vertices, all of them having the mark $R$, and $\{n+1, \ldots, 2n\}$ is
the set of right vertices, all of them having the mark $B$. There are
$3n$ edges in the graph, comprised of the edges $(i, n+ \llbracket i \rrbracket)$, 
$(i,n+ \llbracket i+1 \rrbracket)$,
and $(i,n+ \llbracket i+2 \rrbracket)$ for $1 \le i \le n$, where
for an integer $k$, $\llbracket k \rrbracket$ is defined to be $n$ if $k \mod n
= 0$, and $k \mod n$ otherwise, so that $1 \leq \llbracket k \rrbracket \leq n$.
See Figure~\ref{fig:bip-example_graph} for an example.
The local weak limit of this sequence of graphs gives probability $\frac{1}{2}$
to the equivalence class of each of the two rooted marked infinite graphs
described below. The underlying rooted unmarked infinite graph equivalence class
for each of these two rooted marked equivalence classes is the same and 
can be described as follows: There is a single
vertex at level 0, which is the root, three vertices at level $1$, and four vertices at each
of the levels $m$ for $m \ge 2$. For the purpose of describing the limit
(there is no such numbering in the limit), one can number  the vertex at level
zero as $0$, the three vertices at level $1$ as $(1,1)$, $(1,2)$ and $(1,3)$,
and the four vertices at level $m$, for each $m \ge 2$, as 
$(m,1)$, $(m,2)$, $(m,3)$ and $(m,4)$ such that the edges are the following:
Vertex $0$ is connected to each of the vertices $(1,1)$, $(1,2)$ and
$(1,3)$. Vertex $(1,1)$ is connected to $(2,1)$ and $(2,2)$, vertex
$(1,2)$ is connected to $(2,2)$ and $(2,3)$, and vertex $(1,3)$
is connected to $(2,3)$ and $(2,4)$. The edges between
the vertices at
level $k$ and those at level $k+1$, for $k \ge 2$, are given by the pattern
$((k,1), (k+1,1))$, $((k,1), (k+1,2))$, $((k,2), (k+1,2))$,
$((k,3),(k+1,3))$, $((k,4),(k+1,3))$, $((k,4),(k+1,4))$.
There are no other edges.
As for the distinction between the two rooted marked equivalence
classes which each get probability $\frac{1}{2}$ in the limit, this 
corresponds to the distinction between
choosing the mark $R$ for the root and then alternating between
marks $B$ and $R$ as one moves from level to level, or choosing the 
mark $B$ for the root and then alternating between marks $R$ and $B$
as one moves from level to level. See Figure~\ref{fig:bip-example_limit} for an example.

\begin{figure}
  \centering
  \hfill
  \subfloat[]{
    \label{fig:bip-example_graph}
\centering
      \begin{tikzpicture}[scale=0.8]
        \begin{scope}[xshift=-1.5cm]
          \node[nodeR] (l1) at (0,0) {};
          \node[nodeR] (l2) at (0,-1) {};
          \node[nodeR] (l3) at (0,-2) {};
          \node[nodeR] (l4) at (0,-3) {};
          \node[nodeR] (l5) at (0,-4) {};
          \node[nodeR] (l6) at (0,-5) {};
          \nodelabel{l1}{180}{1};
          \nodelabel{l2}{180}{2};
          \nodelabel{l3}{180}{3};
          \nodelabel{l4}{180}{4};
          \nodelabel{l5}{180}{5};
          \nodelabel{l6}{180}{6};
        \end{scope}
        \begin{scope}[xshift=1.5cm]
          \node[nodeB] (r1) at (0,0) {};
          \node[nodeB] (r2) at (0,-1) {};
          \node[nodeB] (r3) at (0,-2) {};
          \node[nodeB] (r4) at (0,-3) {};
          \node[nodeB] (r5) at (0,-4) {};
          \node[nodeB] (r6) at (0,-5) {};
          \nodelabel{r1}{0}{7};
          \nodelabel{r2}{0}{8};
          \nodelabel{r3}{0}{9};
          \nodelabel{r4}{0}{10};
          \nodelabel{r5}{0}{11};
          \nodelabel{r6}{0}{12};
        \end{scope}
        \draw[thick]
        (l1) -- (r1) (l1) -- (r2) (l1) -- (r3)
        (l2) -- (r2) (l2) -- (r3) (l2) -- (r4)
        (l3) -- (r3) (l3) -- (r4) (l3) -- (r5)
        (l4) -- (r4) (l4) -- (r5) (l4) -- (r6)
        (l5) -- (r5) (l5) -- (r6) (l5) -- (r1)
        (l6) -- (r6) (l6) -- (r1) (l6) -- (r2);
      \end{tikzpicture}
    }%
    \hfill
    \subfloat[]{
      \label{fig:bip-example_limit}
    \begin{tikzpicture}[scale=0.45]
      \node[nodeR] (n0) at (0,0) {};
      \begin{scope}[yshift=-2cm]
        \node[nodeB] (n11) at (-4,0) {};
        \node[nodeB] (n12) at (0,0) {};
        \node[nodeB] (n13) at (4,0) {};
        \begin{scope}[yshift=-2cm]
          \node[nodeR] (n21) at (-6,0) {};
          \node[nodeR] (n22) at (-2,0) {};
          \node[nodeR] (n23) at (2,0) {};
          \node[nodeR] (n24) at (6,0) {};
          \begin{scope}[yshift=-2cm]
            \node[nodeB] (n31) at (-6,0) {};
            \node[nodeB] (n32) at (-2,0) {};
            \node[nodeB] (n33) at (2,0) {};
            \node[nodeB] (n34) at (6,0) {};
            \begin{scope}[yshift=-2cm]
              \node[nodeR] (n41) at (-6,0) {};
              \node[nodeR] (n42) at (-2,0) {};
              \node[nodeR] (n43) at (2,0) {};
              \node[nodeR] (n44) at (6,0) {};
              \begin{scope}[yshift=-0.6cm]
                \node[rotate=90] at (-4,0) {\dots};
                \node[rotate=90] at (4,0) {\dots};
              \end{scope}
            \end{scope}
          \end{scope}
        \end{scope}
      \end{scope}

      \draw[thick]
      (n0) -- (n11) (n0) -- (n12) (n0) -- (n13)
      (n21) -- (n11) -- (n22) -- (n12) -- (n23) -- (n13) -- (n24)
      (n31) -- (n21) -- (n32) -- (n22) (n23) -- (n33) -- (n24) -- (n34)
      (n41) -- (n31) -- (n42) -- (n32) (n33) -- (n43) -- (n34) -- (n44);
    \end{tikzpicture}
  }
  \hfill
  \caption[Bipartite example]{The graph in Example~\ref{item:bipartite-example}, $(a)$ illustrates
    the graph $G_6$ which has $12$ vertices and $18$ edges. The vertex mark set is
    $\vermark = \{{\color{bluenodecolor} B} (\tikz{\node[nodeB] at (0,0) {};}),
    {\color{rednodecolor} R} (\tikz{\node[nodeR] at (0,0)
  {};})\}$, and the edge mark set $\edgemark$ has cardinality 1. The local weak
limit of $G_n$ is a random rooted graph which gives probability $1/2$ to the
rooted marked infinite graph illustrated in  $(b)$, and gives probability $1/2$
to a similar rooted marked 
graph which has a structure identical to $(b)$, but the mark of each vertex is
switched from $R$ to $B$ and vice versa. }
\end{figure}
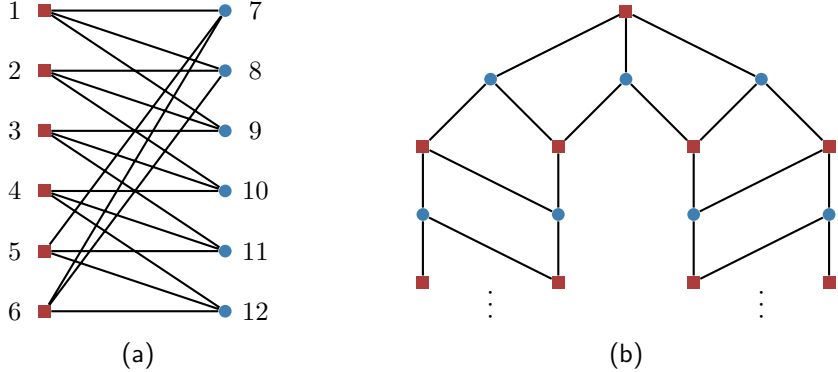


\end{enumerate}

The following lemma
gives a useful tool for establishing when local weak convergence holds. This lemma is proved in Appendix~\ref{sec:proof-equivalence-condition-lwc}.
\begin{lem}
  \label{lem:eq-condition-local-weak-convergence}
 Let $\{ \mu_n \}_{n \geq 1}$ and $\mu$ be Borel probability measures on
 $\mGb_*$ such that the support of $\mu$ is a subset of $\mTb_*$. Then $\mu_n
 \Rightarrow \mu$ iff the following condition is satisfied: For all $h \geq 0$
 and for all rooted marked trees $(T, i)$ with depth at most  $h$, if
  \begin{equation}  \label{eq:Ah}
    A^h_{(T, i)}:= \{ [G, o] \in \mGb_*: (G, o)_h \equiv (T, i) \},
  \end{equation}
then $\mu_n(A^h_{(T, i)}) \rightarrow \mu(A^h_{(T,i)})$.
\end{lem}


\subsection{Unimodularity}
\label{sec:unimodularity}

In order to get a better understanding of the nature of the 
results proved in this paper, it is important to understand what is
meant by a
{\em unimodular} probability distribution $\mu \in \mP(\mGb_*)$.
We give the relevant definitions and context in this section.

Since each vertex in $G_n$ has the same chance of being chosen as 
the root in the definition of $U(G_n)$, this should manifest itself as some
kind of stationarity property of the limit $\mu$, with respect
to changes of the root.
A probability distribution $\mu \in \mP(\mGb_*)$ is called {\em sofic} 
if there exists a sequence of finite graphs $G_n$ with local weak limit $\mu$. 
The definition of unimodularity is made in an attempt to understand
what it means for a Borel probability distribution on $\mGb_*$
to be sofic.

To define unimodularity, let $\mGb_{**}$ be the set of isomorphism classes
$[G,o,v]$ where $G$ is a marked connected graph with two distinguished vertices
$o$ and $v$ in $V(G)$ (ordered, but not necessarily distinct). Here, isomorphism
is defined by an adjacency preserving vertex bijection which preserves
vertex and edge marks, and also maps the two distinguished vertices of one
object to the respective ones of the other. A measure $\mu \in \mP(\mGb_*)$ is
said to be unimodular if, for all measurable functions $f:
\mGb_{**} \rightarrow \reals_+$, we have
\begin{equation}
  \label{eq:unim-integral}
  \int \sum_{v \in V(G)} f([G,o,v]) d\mu([G,o]) = \int \sum_{v \in V(G)} f([G,v,o]) d\mu([G,o]).
\end{equation}
Here, the summation is taken over all vertices $v$ which are in the same
connected component of $G$ as $o$. It can be seen that it suffices to check the
above condition for a function $f$ such that $f([G,o,v]) = 0$ unless $v \sim_G
o$. This is called \emph{involution invariance} \cite{aldous2007processes}.
Let $\mP_u(\mGb_*)$ denote the set of unimodular probability measures on
$\mGb_*$. Also, since $\mTb_* \subset \mGb_*$, we can define the set of
unimodular probability measures on $\mTb_*$ and denote it by $\mP_u(\mTb_*)$. 
A sofic probability measure is unimodular. Whether the other direction also
holds is unknown.


An important consequence of unimodularity is that, roughly speaking, every vertex
has a positive probability to be the root. The following is a rephrasing of
Lemma~2.3 in \cite{aldous2007processes}.

\begin{lem}[Everything Shows at the Root]
  \label{lem:everything-root}
  Let $\mu \in \mP_u(\mGb_*)$ be unimodular. If for a subset $\vermark_0
  \subset \vermark$ the mark at the root is  in $\vermark_0$ almost surely (with $[G,o]$ distributed as $\mu$), then
  the mark at every vertex is  in $\vermark_0$ almost surely. 
  Furthermore, if for
  a subset $A \subset \edgemark \times \edgemark$ it holds that for every vertex $v$ adjacent to the root $o$ the pair of edge marks
  $(\xi_G(v,o), \xi_G(o,v))$ on the edge connecting $o$ to $v$ is in $A$ almost surely (with $[G,o]$ distributed as $\mu$), then for every edge $(u,w)$ the pair of edge marks 
  $(\xi_G(u,w), \xi_G(w,u))$ is in $A$ almost surely.
\end{lem}

\section{A Notion of Entropy for Processes on Random Rooted Marked Graphs}
\label{sec:noti-entr-proc}

In this section, we introduce a notion of entropy for probability distributions
$\mu \in \mP(\mGb_*)$ with $0 < \deg(\mu) < \infty$. This is a generalization
of the notion of entropy introduced by Bordenave and Caputo
in \cite{bordenave2014large} to the marked framework we discuss in this paper,
where vertices and edges are allowed to carry marks. This generalization is due
to us, and the reader is referred to \cite{delgosha2019notion} for more details.

We first need to set up some notation. Throughout the discussion, we assume that the vertex and edge mark sets,
$\vermark$ and $\edgemark$ respectively, are fixed and finite. 
For a marked graph $G$, we define the \emph{edge mark count vector} of $G$ by
$\vm_G := (m_G(x,x'): x,x' \in \edgemark)$ where $m_G(x,x')$ for $x, x' \in
\edgemark$ is the number of edges in $G$ with the pair of marks $x$ and $x'$, i.e.\
the number of edges $(v,w)$ such that $(\xi_G(v,w), \xi_G(w,v)) = (x,x')$ or
$(\xi_G(w,v), \xi_G(v,w)) = (x,x')$. Also, we define the \emph{vertex mark count
vector} of $G$ by $\vu_G := (u_G(\theta): \theta \in \vermark)$ where
$u_G(\theta)$ denotes the number of vertices in $G$ with mark $\theta$.

We define an \emph{edge mark count vector} to be a vector of
nonnegative integers $\vm := (m(x,x'): x, x' \in \edgemark)$ such that $m(x,x')
= m(x',x)$ for all $x, x' \in \edgemark$. Likewise, a \emph{vertex mark count
  vector} is defined to be a vector of nonnegative integers $\vu := (u(\theta):
\theta \in \vermark)$. Since $\edgemark$ is finite, we may assume that it is an
ordered set. We define $\snorm{\vm}_1 := \sum_{x \leq x' \in \edgemark} m(x,x')$
and $\snorm{\vu} := \sum_{\theta \in \vermark} u(\theta)$.

Given $n \in \nats$ together with edge and vertex mark count vectors $\vm$ and
$\vu$ respectively, let $\mGn_{\vm, \vu}$ denote the set of marked graphs $G$ on the vertex
set $\{1, \dots, n\}$ such that $\vm_G = \vm$ and $\vu_G = \vu$. Note that
$\mGn_{\vm, \vu}$ is empty unless $\snorm{\vu}_1 = n$ and $\snorm{\vm}_1 \leq
\binom{n}{2}$.

We define an \emph{average degree vector} to be a vector of nonnegative reals
$\vd := (d_{x,x'} : x, x' \in \edgemark)$ such that, for all $x, x' \in
\edgemark$, we have $d_{x,x'} = d_{x',x}$, and also $\sum_{x,x' \in \edgemark}
d_{x,x'} > 0$.

\begin{definition}
\label{def:deg-seq-adapt}  

  Given an average degree vector $\vd$ and a probability distribution $Q = (q_\theta: \theta \in \vermark)$, we say that a sequence 
$(\vmn,\vun)$, comprised of edge mark count vectors and vertex mark count vectors $\vmn$ and $\vun$ respectively, is adapted to $(\vd, Q)$, if the following conditions hold:
  \begin{enumerate}
  \item For each $n$, we have $\snorm{\vmn}_1 \leq \binom{n}{2}$ and $\snorm{\vun}_1 = n$;
  \item For $x \in \edgemark$, we have $\mn(x,x) / n \rightarrow d_{x,x}/2$;
  \item For $x \neq x' \in \edgemark$, we have $\mn(x, x') / n \rightarrow d_{x,x'} = d_{x',x}$;
  \item For $\theta \in \vermark$, we have $\un(\theta) / n \rightarrow q_\theta$;
  \item For $x, x' \in \edgemark$, $d_{x,x'} = 0$ implies $\mn(x, x') = 0$ for all $n$;
  \item For $\theta \in \vermark$, $q_\theta  = 0$ implies $\un(\theta) = 0$ for all $n$.
  \end{enumerate}
\end{definition}

If $\vmn$ and $\vun$ are sequences such that $(\vmn,\vun)$ is adapted to $(\vd,
Q)$, using Stirling's approximation one can show that 
  \begin{equation}
    \label{eq:log-mGnmnun-Stirling}
    \log | \mGn_{\vmn, \vun} | =  \snorm{\vmn}_1 \log n + n H(Q)  + n \sum_{x,x' \in \edgemark} s(d_{x,x'}) + o(n),
  \end{equation}
  where
  \begin{equation*}
    s(d) :=
    \begin{cases}
      \frac{d}{2} - \frac{d}{2} \log d & d > 0, \\
      0 & d = 0.
    \end{cases}
  \end{equation*}
To simplify the notation, we may write $s(\vd)$ for $\sum_{x,x' \in \edgemark}
s(d_{x,x'})$. 

To give the definition of the BC entropy, we first define the  upper and the lower BC
entropy.
\begin{definition}
  \label{def:BC-entropy}
  
Assume $\mu \in \mP(\mGb_*)$ is given, with $0 < \deg(\mu) < \infty$. For $\epsilon>0$, and edge and vertex mark count vectors
$\vm$ and $\vu$ respectively, define
  \begin{equation*}
    \mGn_{\vm, \vu} (\mu, \epsilon) := \{ G \in \mGn_{\vm, \vu}: \dlp(U(G), \mu) < \epsilon \}.
  \end{equation*}
  Fix an average degree vector $\vd$ and a probability distribution $Q = (q_\theta:
  \theta \in \vermark)$, and also fix sequences of edge and vertex mark
  count vectors $\vmn$ and $\vun$ respectively such that $(\vmn,\vun)$ is adapted to $(\vd, Q)$. With these, define
  \begin{equation*}
    \bchover_{\vd, Q}(\mu, \epsilon)\condmnun := \limsup_{n \rightarrow \infty} \frac{\log |\mGn_{\vmn, \vun}(\mu, \epsilon)| - \snorm{\vmn}_1 \log n}{n},
  \end{equation*}
which we call the $\epsilon$--upper BC entropy. Since this is increasing
in $\epsilon$, we can define the {\em upper BC entropy} as 
  \begin{equation*}
    \bchover_{\vd, Q}(\mu)\condmnun := \lim_{\epsilon \downarrow 0} \bchover_{\vd, Q}(\mu, \epsilon)\condmnun.
  \end{equation*}
We may define the $\epsilon$--lower BC entropy $\bchunder_{\vd, Q}(\mu,
\epsilon)\condmnun$ similarly as
\begin{equation*}
    \bchunder_{\vd, Q}(\mu, \epsilon)\condmnun := \liminf_{n \rightarrow \infty} \frac{\log |\mGn_{\vmn, \vun}(\mu, \epsilon)| - \snorm{\vmn}_1 \log n}{n}.
  \end{equation*}
Since this is increasing
in $\epsilon$, we can define the {\em lower BC entropy} $\bchunder_{\vd, Q}(\mu)\condmnun$ as
\begin{equation*}
    \bchunder_{\vd, Q}(\mu)\condmnun := \lim_{\epsilon \downarrow 0} \bchunder_{\vd, Q}(\mu, \epsilon)\condmnun.
  \end{equation*}
 
    \end{definition}

Now, we state the following properties of the upper and lower marked BC entropy,
which will lead to the definition of the marked BC entropy. The reader is referred to
\cite{delgosha2019notion} for a proof and more details.

\begin{thm}[Theorem 1 in \cite{delgosha2019notion}]
  \label{thm:badcases}
  Let an average degree vector $\vd = (d_{x,x'} : x,x' \in \edgemark)$ and a
  probability distribution $Q = (q_\theta: \theta \in \vermark)$ be given. Suppose $\mu \in \mP(\mGb_*)$ with
 $0 < \deg(\mu) < \infty$ satisfies any one of the following conditions:
 \begin{enumerate}
    \item $\mu$ is not unimodular;
    \item $\mu$ is not supported on $\mTb_*$;
    \item $\deg_{x,x'}(\mu) \neq d_{x,x'}$ for some $x,x' \in \edgemark$, or $\vtype_\theta(\mu) \neq q_\theta$ for some $\theta \in \vermark$.
    \end{enumerate}
    Then, for any choice of the
    sequences $\vmn$ and $\vun$ such that $(\vmn,\vun)$ is adapted to $(\vd, Q)$, we have $\bchover_{\vd, Q}(\mu)\condmnun = -\infty$. 
  \end{thm}

A consequence of Theorem~\ref{thm:badcases}
is that the only case of interest in the discussion of marked
BC entropy is when $\mu \in \mP_u(\mTb_*)$,
$\vd = \vdeg(\mu)$, $Q = \vvtype(\mu)$,
and the
sequences $\vmn$ and $\vun$ are such that $(\vmn,\vun)$ is adapted to
$(\vdeg(\mu), \vvtype(\mu))$.
Namely, the only upper and lower marked BC entropies of interest are 
$\bchover_{\vdeg(\mu), \vvtype(\mu)}(\mu)\condmnun$ and $\bchunder_{\vdeg(\mu), \vvtype(\mu)}(\mu)\condmnun$ respectively.

The following Theorem~\ref{thm:bch-properties} establishes that the upper and lower
marked BC entropies do not depend on the 
choice of the defining pair of sequences 
$(\vmn,\vun)$. Further, 
this theorem establishes that
the upper marked BC entropy 
is always equal to the lower marked BC entropy. The reader is referred to
\cite{delgosha2019notion} for a proof and more details. 

\begin{thm}[Theorem 2 in \cite{delgosha2019notion}]
  \label{thm:bch-properties}
  Assume that an average degree vector $\vd = (d_{x,x'} : x,x' \in \edgemark)$ together with a
  probability distribution $Q = (q_\theta: \theta \in \vermark)$ are given. For
  any  $\mu \in \mP(\mGb_*)$ such that
 $0 < \deg(\mu) < \infty$, we have 
  \begin{enumerate}
  \item \label{thm:BC-invariant} The values of $\bchover_{\vd, Q}(\mu)\condmnun$ and
    $\bchunder_{\vd, Q}(\mu)\condmnun$ are invariant under the specific choice of the
    sequences $\vmn$ and $\vun$ such that $(\vmn,\vun)$ is adapted to $(\vd, Q)$. With this,
    we may simplify the notation and unambiguously write $\bchover_{\vd, Q}(\mu)$ and
    $\bchunder_{\vd, Q}(\mu)$. 
  \item \label{thm:BC-well} 
  $\bchover_{\vd, Q}(\mu) = \bchunder_{\vd, Q}(\mu)$. 
  We may therefore unambiguously write $\bch_{\vd, Q}(\mu)$ 
for this common value,
and call it the {\em marked BC entropy} of 
$\mu \in \mP(\mGb_*)$ for the 
average degree vector $\vd$ and a probability distribution $Q = (q_\theta:
  \theta \in \vermark)$.
  Moreover, $\bch_{\vd, Q}(\mu) \in [-\infty, s(\vd) + H(Q)]$.
  \end{enumerate}
\end{thm}

From Theorem~\ref{thm:badcases} we conclude that unless 
$\vd = \vdeg(\mu)$, $Q = \vvtype(\mu)$, and $\mu$
  is a unimodular measure on $\mTb_*$, we have 
  $\bch_{\vd, Q}(\mu) = -\infty$. 
  In view of this, for $\mu \in \mP(\mGb_*)$
  with $0< \deg(\mu) < \infty$, we write $\bch(\mu)$
  for  $\bch_{\vdeg(\mu), \vvtype(\mu)}(\mu)$. Likewise, we may write
  $\bchunder(\mu)$ and $\bchover(\mu)$ for $\bchunder_{\vdeg(\mu),
    \vvtype(\mu)}(\mu)$ and $\bchover_{\vdeg(\mu),
    \vvtype(\mu)}(\mu)$, respectively. 
    Note that, unless $\mu \in \mP_u(\mTb_*)$, 
    we have $\bchover(\mu) = \bchunder(\mu) = \bch(\mu) = -\infty$.
    
We are now in a position to define the marked BC entropy.

\begin{definition}
  \label{def:BC-entropy-new}
  For $\mu \in \mP(\mGb_*)$
  with $0 < \deg(\mu) < \infty$, the marked BC entropy of $\mu$ is defined to be $\bch(\mu)$.
\end{definition}

\section{Main Results}
\label{sec:main-results}

Let $\mGb_n$ be the set of marked graphs on the vertex set  $\{1, \dots, n\}$, with 
edge marks from $\edgemark$ and vertex marks from $\vermark$.
Our goal is to design a 
compression scheme, comprised of compression and decompression functions $f_n$ and $g_n$ for each $n$, such that $f_n$ maps $\mGb_n$ to $\{0,1\}^* - \emptyset$ and $g_n$ maps $\{0,1\}^* - \emptyset$ to $\mGb_n$, with the condition that $g_n \circ f_n(G) = G$ for all $G \in \mGb_n$. 
Motivated by the notion of entropy introduced in the previous section, 
we want our compression scheme to be universally optimal in the following sense: if $\mu
\in \mP(\mTb_*)$ is unimodular and $\Gn$ is a sequence of marked graphs with
local weak limit $\mu$,
then, with $\vmn := \vm_{\Gn}$,
we have 
\begin{equation}    \label{eq:scheme}
  \limsup_{n \rightarrow \infty} \frac{\nat(f_n(\Gn)) - \snorm{\vmn}_1 \log n}{n} \leq \bch(\mu).
\end{equation}
In Section~\ref{sec:univ-coding-algor}, we design 
such a universally optimal compression scheme and prove its optimality. 
This is stated formally in the next theorem.

\begin{thm}
  \label{thm:universal-existence-main}
There is a  compression scheme that is optimal in the above sense
for all $\mu \in \mP_u(\mTb_*)$ such that $\deg(\mu) \in (0,\infty)$.
\end{thm}

We also prove the following converse theorem, which justifies the claim of optimality for compression schemes that satisfy the condition in \eqref{eq:scheme}.

\begin{thm}
  \label{thm:converse}
Assume that 
a compression scheme
$\{f_n, g_n\}_{n=1}^\infty$ is given. Fix some unimodular $\mu \in \mP(\mTb_*)$
such that  $\deg(\mu) \in (0,\infty)$ and $\bch(\mu) > -\infty$. Moreover, fix a
sequence $\vmn$  and $\vun$ of edge mark count vectors and vertex mark count
vectors respectively, such that $(\vmn, \vun)$ is adapted to $(\vdeg(\mu), \vvtype(\mu))$.  Then, there exists a sequence of positive real numbers $\epsilon_n$
going to zero, together with a sequence of independent graph--valued random
variables 
$\{\Gn\}_{n=1}^\infty$ defined on a joint probability space, with $\Gn$ being
uniform in $\mG^{(n)}_{\vmn, \vun}(\mu, \epsilon_n)$, such that with
probability  one
\begin{equation*}
  \liminf_{n \rightarrow \infty} \frac{\nat(f_n(\Gn)) - \snorm{\vmn}_1\log n}{n} \geq \bch(\mu).
\end{equation*}
\end{thm}

\begin{proof}
First note that any marked graph on $n$ vertices can be represented with $O(n^2)$ bits. Hence, without loss of generality, we may assume that,
for some finite positive constant $c$, we have $\nat(f_n(\Gn)) \leq c n^2$ for all $\Gn$ on $n$ vertices. Consequently, by adding a header of size $O(\log n^2) = O(\log n)$ to the beginning of each codeword in $f_n$, in order to describe its length, we can make $f_n$ prefix--free. 
Thus, without loss of generality, we may assume that $f_n$ is prefix--free for all $n$.

  From the definition of $\bchunder(\mu)$, one can find a sequence of positive numbers $\epsilon_n$ going to zero, such that
  \begin{equation*}
    \bchunder(\mu) = \liminf_{n \rightarrow \infty} \frac{ \log |\mG^{(n)}_{\vmn, \vun}(\mu, \epsilon_n)| - \snorm{\vmn}_1 \log n}{n}.
  \end{equation*}
From Theorem~\ref{thm:bch-properties}, we have $\bchunder(\mu) = \bch(\mu)$, and since $\bch(\mu) > -\infty$ by assumption, $\mG^{(n)}_{\vmn, \vun}(\mu, \epsilon_n)$ is nonempty once $n$ is large enough. Using Kraft's inequality, we have
\[
\sum_{G \in \mG^{(n)}_{\vmn, \vun}(\mu, \epsilon_n)} e^{-\nat(f_n(G))} \leq 1. 
\]
With $\Gn$ being uniform in $\mG^{(n)}_{\vmn, \vun}(\mu, \epsilon_n)$, the Markov inequality then implies that
\begin{equation*}
  \pr{\nat(f_n(\Gn)) < \log |\mG^{(n)}_{\vmn, \vun}(\mu, \epsilon_n)| - 2\log n} \leq \frac{1}{n^2}.
\end{equation*}
From this, using the Borel-Cantelli lemma, we have $\nat(f_n(\Gn)) \geq \log |\mG^{(n)}_{\vmn, \vun}(\mu, \epsilon_n)| - 2\log n$ eventually. Therefore, with probability $1$, we have 
\begin{equation*}
\small
\begin{split}
  \liminf_{n \rightarrow \infty} \frac{\nat(f_n(\Gn)) - \snorm{\vmn}_1 \log n}{n} &\geq \liminf_{n \rightarrow \infty}  \frac{\log |\mG^{(n)}_{\vmn, \vun}(\mu, \epsilon_n)| - \snorm{\vmn}_1\log n}{n} \\
  &= \bchunder(\mu)\\
  &= \bch(\mu),
\end{split}
\end{equation*}
which completes the proof.
\end{proof}


\section{The Universal Compression Scheme}
\label{sec:univ-coding-algor}

In this section, we propose our compression scheme.  First, in Section~\ref{sec:coding-step-1-restriction-max-degree}, we introduce our compression scheme under certain assumptions. Then, in Section~\ref{sec:generel-scheme}, we relax these assumptions.

\subsection{A First--step Scheme}
\label{sec:coding-step-1-restriction-max-degree}


We first give an outline of the compression scheme, then illustrate it via an example, and finally formally 
describe it
and prove its optimality. 
Fix two sequences of integers $k_n$ and $\Delta_n$ as design parameters, which will be specified in Section~\ref{sec:generel-scheme}.
Given a marked graph $\Gn$ on $n$ vertices, 
with maximum degree no more than $\Delta_n$,
we first encode its depth--$k_n$ empirical distribution, i.e. $U_{k_n}(\Gn)$
(defined in \eqref{eq:UkG}). We do this by counting the number of times each
marked rooted graph with depth at most $k_n$ and maximum degree at most
$\Delta_n$ appears in the graph $\Gn$. Then, in the set of all graphs which result
in these counts, we specify the target graph $\Gn$. Figure~\ref{fig:alg-example}
illustrates an example of this procedure. In this example, the marked graph on
$n=4$ vertices in Figure~\ref{fig:alg-exm-graph} is given and the design
parameters $k_n = 1$ and $\Delta_n = 2$ are chosen. We then list all the rooted
marked graphs with depth at most $k_n = 1$ and maximum degree at most $\Delta_n=
2$, and count the number of times each of these patterns appears in the graph, as
depicted in Figure~\ref{fig:alg-exm-Akdelta}. Finally, we consider all the graphs that
would result in the same counts if we run this procedure on them (shown in
Figure~\ref{fig:alg-exm-Wn} for this example), and specify the input graph
within this collection of graphs. In principle, this scheme is similar to the
conventional universal coding for sequential data in which we first specify the
type of a given sequence and then specify the sequence itself among all the
sequences that have this  type. 

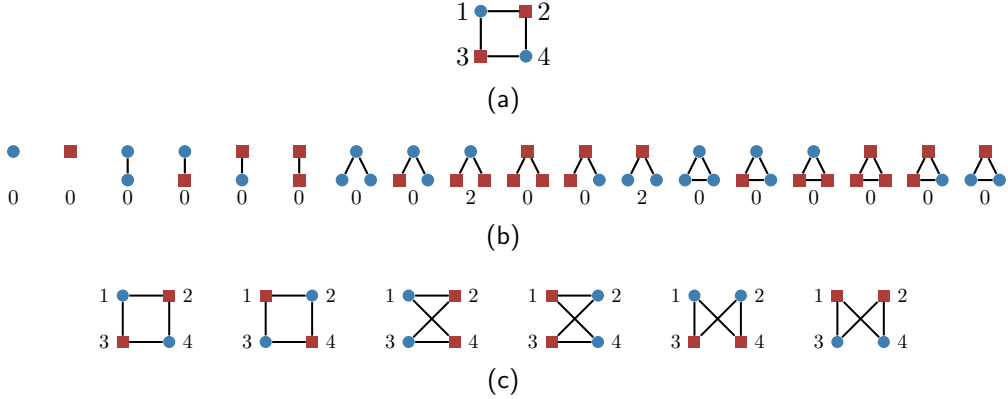
\begin{figure}
  \centering
  \subfloat[]{
    \label{fig:alg-exm-graph}
  \begin{tikzpicture}[scale=0.6]
        \begin{scope}[scale=0.5]
          \node[nodeB] (n1) at (0,0) {};
          \node[nodeR] (n2) at (2,0) {};
          \node[nodeR] (n3) at (0,-2) {};
          \node[nodeB] (n4) at (2,-2) {};

          \draw[thick] (n1) -- (n2) -- (n4) -- (n3) -- (n1);

          \nodelabeldist{n1}{180}{8mm}{1};
          \nodelabeldist{n2}{0}{8mm}{2};
          \nodelabeldist{n3}{180}{8mm}{3};
          \nodelabeldist{n4}{0}{8mm}{4};
    \end{scope}
  \end{tikzpicture}
}

\subfloat[]{
  \label{fig:alg-exm-Akdelta}
\begin{tikzpicture}[scale=0.76]
  \begin{scope}
    \node[nodeB] at (0,0) {};
  \end{scope}

  \begin{scope}[xshift=1cm,scale=0.5]
    \node[nodeR] at (0,0) {};
  \end{scope}

  \begin{scope}[xshift=2cm,scale=0.5]
    \node[nodeB] (n1) at (0,0) {};
    \node[nodeB] (n2) at (0,-1) {};
    \draw[thick] (n1) -- (n2);
  \end{scope}

  \begin{scope}[xshift=3cm,scale=0.5]
    \node[nodeB] (n1) at (0,0) {};
    \node[nodeR] (n2) at (0,-1) {};
    \draw[thick] (n1) -- (n2);
  \end{scope}

\begin{scope}[xshift=4cm,scale=0.5]
    \node[nodeR] (n1) at (0,0) {};
    \node[nodeB] (n2) at (0,-1) {};
    \draw[thick] (n1) -- (n2);
  \end{scope}

  \begin{scope}[xshift=5cm,scale=0.5]
    \node[nodeR] (n1) at (0,0) {};
    \node[nodeR] (n2) at (0,-1) {};
    \draw[thick] (n1) -- (n2);
  \end{scope}

\begin{scope}[xshift=6cm,scale=0.5]
  \node[nodeB] (n1) at (0,0) {};
  \node[nodeB] (n2) at (-0.5,-1) {};
  \node[nodeB] (n3) at (0.5,-1) {};
  \draw[thick] (n1) -- (n2) (n1) -- (n3);
  \end{scope}

\begin{scope}[xshift=7cm,scale=0.5]
  \node[nodeB] (n1) at (0,0) {};
  \node[nodeR] (n2) at (-0.5,-1) {};
  \node[nodeB] (n3) at (0.5,-1) {};
  \draw[thick] (n1) -- (n2) (n1) -- (n3);
  \end{scope}

\begin{scope}[xshift=8cm,scale=0.5]
  \node[nodeB] (n1) at (0,0) {};
  \node[nodeR] (n2) at (-0.5,-1) {};
  \node[nodeR] (n3) at (0.5,-1) {};
  \draw[thick] (n1) -- (n2) (n1) -- (n3);
  \end{scope}

\begin{scope}[xshift=9cm,scale=0.5]
  \node[nodeR] (n1) at (0,0) {};
  \node[nodeR] (n2) at (-0.5,-1) {};
  \node[nodeR] (n3) at (0.5,-1) {};
  \draw[thick] (n1) -- (n2) (n1) -- (n3);
  \end{scope}

\begin{scope}[xshift=10cm,scale=0.5]
  \node[nodeR] (n1) at (0,0) {};
  \node[nodeR] (n2) at (-0.5,-1) {};
  \node[nodeB] (n3) at (0.5,-1) {};
  \draw[thick] (n1) -- (n2) (n1) -- (n3);
  \end{scope}

\begin{scope}[xshift=11cm,scale=0.5]
  \node[nodeR] (n1) at (0,0) {};
  \node[nodeB] (n2) at (-0.5,-1) {};
  \node[nodeB] (n3) at (0.5,-1) {};
  \draw[thick] (n1) -- (n2) (n1) -- (n3);
  \end{scope}


  \begin{scope}[xshift=12cm,scale=0.5]
  \node[nodeB] (n1) at (0,0) {};
  \node[nodeB] (n2) at (-0.5,-1) {};
  \node[nodeB] (n3) at (0.5,-1) {};
  \draw[thick] (n1) -- (n2) -- (n3) -- (n1);
  \end{scope}

\begin{scope}[xshift=13cm,scale=0.5]
  \node[nodeB] (n1) at (0,0) {};
  \node[nodeR] (n2) at (-0.5,-1) {};
  \node[nodeB] (n3) at (0.5,-1) {};
    \draw[thick] (n1) -- (n2) -- (n3) -- (n1);

  \end{scope}

\begin{scope}[xshift=14cm,scale=0.5]
  \node[nodeB] (n1) at (0,0) {};
  \node[nodeR] (n2) at (-0.5,-1) {};
  \node[nodeR] (n3) at (0.5,-1) {};
      \draw[thick] (n1) -- (n2) -- (n3) -- (n1);
  \end{scope}

\begin{scope}[xshift=15cm,scale=0.5]
  \node[nodeR] (n1) at (0,0) {};
  \node[nodeR] (n2) at (-0.5,-1) {};
  \node[nodeR] (n3) at (0.5,-1) {};
      \draw[thick] (n1) -- (n2) -- (n3) -- (n1);
  \end{scope}

\begin{scope}[xshift=16cm,scale=0.5]
  \node[nodeR] (n1) at (0,0) {};
  \node[nodeR] (n2) at (-0.5,-1) {};
  \node[nodeB] (n3) at (0.5,-1) {};
      \draw[thick] (n1) -- (n2) -- (n3) -- (n1);
  \end{scope}

\begin{scope}[xshift=17cm,scale=0.5]
  \node[nodeR] (n1) at (0,0) {};
  \node[nodeB] (n2) at (-0.5,-1) {};
  \node[nodeB] (n3) at (0.5,-1) {};
      \draw[thick] (n1) -- (n2) -- (n3) -- (n1);
  \end{scope}

  \foreach \x in {0,...,7}{
    \node[scale=0.8] at (\x,-0.8) {0};}
  \foreach \x in {9,...,10}{
    \node[scale=0.8] at (\x,-0.8) {0};}
  \foreach \x in {12,...,17}{
    \node[scale=0.8] at (\x,-0.8) {0};}

  \node[scale=0.8] at (8,-0.8) {2};
  \node[scale=0.8] at (11,-0.8) {2};
\end{tikzpicture}

}

\subfloat[]{
\label{fig:alg-exm-Wn}
\scalebox{0.95}{
  \begin{tikzpicture}
    \begin{scope}[scale=0.65]
      \node[nodeB] (n1) at (0,0) {};
      \node[nodeR] (n2) at (1,0) {};
      \node[nodeR] (n3) at (0,-1) {};
      \node[nodeB] (n4) at (1,-1) {};

      \draw[thick] (n1) -- (n2) -- (n4) -- (n3) -- (n1);
      \node[scale=0.8] at ($(n1) + (180:4mm)$) {1};
      \node[scale=0.8] at ($(n2) + (0:4mm)$) {2};
      \node[scale=0.8] at ($(n3) + (180:4mm)$) {3};
      \node[scale=0.8] at ($(n4) + (0:4mm)$) {4};
      
    \end{scope}

    \begin{scope}[xshift=2cm,scale=0.65]
      \node[nodeR] (n1) at (0,0) {};
      \node[nodeB] (n2) at (1,0) {};
      \node[nodeB] (n3) at (0,-1) {};
      \node[nodeR] (n4) at (1,-1) {};

      \draw[thick] (n1) -- (n2) -- (n4) -- (n3) -- (n1);
      \node[scale=0.8] at ($(n1) + (180:4mm)$) {1};
      \node[scale=0.8] at ($(n2) + (0:4mm)$) {2};
      \node[scale=0.8] at ($(n3) + (180:4mm)$) {3};
      \node[scale=0.8] at ($(n4) + (0:4mm)$) {4};
    \end{scope}

    \begin{scope}[xshift=4cm,scale=0.65]
      \node[nodeB] (n1) at (0,0) {};
      \node[nodeR] (n2) at (1,0) {};
      \node[nodeB] (n3) at (0,-1) {};
      \node[nodeR] (n4) at (1,-1) {};

      \draw[thick] (n1) -- (n2) -- (n3) -- (n4) -- (n1);
      \node[scale=0.8] at ($(n1) + (180:4mm)$) {1};
      \node[scale=0.8] at ($(n2) + (0:4mm)$) {2};
      \node[scale=0.8] at ($(n3) + (180:4mm)$) {3};
      \node[scale=0.8] at ($(n4) + (0:4mm)$) {4};

    \end{scope}

   \begin{scope}[xshift=6cm,scale=0.65]
      \node[nodeR] (n1) at (0,0) {};
      \node[nodeB] (n2) at (1,0) {};
      \node[nodeR] (n3) at (0,-1) {};
      \node[nodeB] (n4) at (1,-1) {};

      \draw[thick] (n1) -- (n2) -- (n3) -- (n4) -- (n1);
      \node[scale=0.8] at ($(n1) + (180:4mm)$) {1};
      \node[scale=0.8] at ($(n2) + (0:4mm)$) {2};
      \node[scale=0.8] at ($(n3) + (180:4mm)$) {3};
      \node[scale=0.8] at ($(n4) + (0:4mm)$) {4};

    \end{scope}

   \begin{scope}[xshift=8cm,scale=0.65]
      \node[nodeB] (n1) at (0,0) {};
      \node[nodeB] (n2) at (1,0) {};
      \node[nodeR] (n3) at (0,-1) {};
      \node[nodeR] (n4) at (1,-1) {};

      \draw[thick] (n1) -- (n3) -- (n2) -- (n4) -- (n1);
      \node[scale=0.8] at ($(n1) + (180:4mm)$) {1};
      \node[scale=0.8] at ($(n2) + (0:4mm)$) {2};
      \node[scale=0.8] at ($(n3) + (180:4mm)$) {3};
      \node[scale=0.8] at ($(n4) + (0:4mm)$) {4};

    \end{scope}

    \begin{scope}[xshift=10cm,scale=0.65]
      \node[nodeR] (n1) at (0,0) {};
      \node[nodeR] (n2) at (1,0) {};
      \node[nodeB] (n3) at (0,-1) {};
      \node[nodeB] (n4) at (1,-1) {};

      \draw[thick] (n1) -- (n3) -- (n2) -- (n4) -- (n1);
      \node[scale=0.8] at ($(n1) + (180:4mm)$) {1};
      \node[scale=0.8] at ($(n2) + (0:4mm)$) {2};
      \node[scale=0.8] at ($(n3) + (180:4mm)$) {3};
      \node[scale=0.8] at ($(n4) + (0:4mm)$) {4};

    \end{scope}

  \end{tikzpicture}
}
}
  \caption[Example of coding scheme]{\label{fig:alg-example} An example of encoding via the compression function associated to our compression scheme with the
    parameter $k=1$ and  the graph $G_4$ on $n=4$ vertices, with vertex mark set $\vermark = \{\tikz{\node[nodeB] at (0,0) {};}, \tikz{\node[nodeR] at (0,0)
  {};}\}$ and edge mark set $\edgemark$ with cardinality 1, shown for (a) acting
    as the input. (b) depicts all members in the set $\mA_{1,2}$ with the
    corresponding number of times each of them appears in the graph, i.e. the
    vector $\psi_{G}$ and (c) illustrates all the graphs with the same count
    vector, i.e.  $W_4$.}
\end{figure}

Before formally explaining the compression scheme, we need some definitions. For integers $k$ and $\Delta$, let $\mA_{k, \Delta}$ be the set of rooted marked graphs $[G, o] \in \mGb_*$ with depth at most $k$ and maximum degree at most $\Delta$. Note that since $k$ and $\Delta$ are finite and the mark sets  are also finite sets, $\mA_{k, \Delta}$ is  a finite set.  

For a marked graph $\Gn$ on the vertex set $\{1, \dots, n\}$, with maximum degree at most $\Delta_n$, and for $[G, o] \in \mA_{k_n, \Delta_n}$, we denote the set 
$\{ 1 \leq i \leq n : [\Gn, i]_{k_n} = [G, o] \}$ 
by $\psi^{(n)}_{\Gn}([G, o])$.
This is  the set of vertices in $\Gn$ with local structure $[G, o]$ up to depth $k_n$. 
Recall that 
$[\Gn, i]_{k_n} = [G, o]$ 
means that $\Gn$ rooted at $i$ is isomorphic to $(G, o)$ up to 
depth $k_n$.
Note that when the maximum degree in $\Gn$ is no more than $\Delta_n$, $[\Gn,i]_{k_n}$ is a member of $\mA_{k_n, \Delta_n}$, for all $1 \leq i \leq n$. Therefore, the subsets $\psi_{\Gn}^{(n)}([G,o])$,
as $[G,o]$ ranges over $\mA_{k_n, \Delta_n}$, form a partition of $\{1, \dots, n\}$. 

We encode a marked graph $\Gn$ with vertex set $\{1, \dots, n\}$ and maximum degree no more than $\Delta_n$ as follows:
\begin{enumerate}
\item Encode the vector $(|\psi^{(n)}_{\Gn}([G,o])|, [G,o] \in \mA_{k_n,
    \Delta_n} )$. Since we have $|\psi^{(n)}_{\Gn}([G,o])|  \leq n$ for all $ [G,o] \in \mA_{k_n, \Delta_n} $, this can be done with at most 
  $|\mA_{k_n, \Delta_n}| 
  (1 + \lfloor  \log_2 n \rfloor)
  \log 2$ nats.
  \item Let $W_n$ be the set of marked graphs $G$ on the vertex set $\{1, \dots, n \}$ with degrees bounded by $\Delta_n$ such that 
    \begin{equation}
\label{eq:Wn-condiiton}
 | \psi^{(n)}_{G}([G', o']) | = |\psi_{\Gn}^{(n)}([G', o'])|,   \qquad    \forall\,\, [G', o'] \in \mA_{k_n, \Delta_n}.
    \end{equation}
Specify $\Gn$ among the elements of $W_n$ 
by sending   $(1+\lfloor  \log_2 |W_n| \rfloor) \log 2$ nats 
to the decoder.
\end{enumerate}


See Figure~\ref{fig:alg-example} for an example of the running of this procedure. 

\begin{rem}
\label{rem:first-step-local-queries}
  The vector $(\psi_{\Gn}^{(n)}([G, o]): [G, o] \in \mA_{k_n, \Delta_n})$ is directly compressed in the above scheme; therefore, we are capable of making local queries in the compressed form without going through the decompression process. An example of such a query is ``how many triangles exist in the graph?''
\end{rem}

Now we show the optimality of this compression scheme under an assumption on
$k_n$ and $\Delta_n$ that allows us to bound the size of the set $\mA_{k_n, \Delta_n}$. 
Lemma \ref{lem:A-k-delta-o(n/logn)}, which is proved in 
Appendix \ref{sec:proofs-univ-coding-algorithm}, shows that the
assumptions made in Proposition
\ref{prop:universal-opt-bounded-degree} below are not vacuous.
To avoid confusion, we use $\tilde{f}_n$ for the compression function in this section and $f_n$ for that of Section~\ref{sec:generel-scheme} (which
does not require any a priori assumed bound on the maximum degree
of the graph on $n$ vertices, $\Gn$).

\begin{prop}
\label{prop:universal-opt-bounded-degree}
  Fix the parameters $k_n$ and $\Delta_n$ 
  such that $|\mA_{k_n,
    \Delta_n}| = o ( \frac{n}{\log n})$, and $k_n \rightarrow \infty$ as
  $n\rightarrow \infty$. Assume that a  sequence of marked graphs
  $\{\Gn\}_{n=1}^\infty$ is given such that $\Gn$ is on the vertex set $\{1,
  \dots, n\}$, the  maximum degree  in $\Gn$ is no
  more than $\Delta_n$, and $\{\Gn\}_{n=1}^\infty$ converges in the local weak sense to a
  unimodular $\mu \in
  \mP_u(\mTb_*)$ as $n \to \infty$. Furthermore, assume that $\deg(\mu) \in (0,\infty)$.
  Then we have 
  \begin{equation}
\label{eq:univ-coding-optimality-prop}
      \limsup_{n \rightarrow \infty} \frac{\nat(\tilde{f}_n(\Gn)) - \snorm{\vmn}_1 \log n}{n} \leq \bch(\mu),
    \end{equation}
    where $\vmn := \vm_{\Gn}$. 
\end{prop}

Before proving this proposition, we need the following tools.
Lemma~\ref{lem:local-isomorphism-LP-distance} is stated in a way which is
stronger than what we need here, but this stronger form will prove useful later
on. The proofs of
Lemmas~\ref{lem:local-isomorphism-LP-distance},~\ref{lem:upper-bound-on-G_n-m_mlogn-n},
and
\ref{lem:log-Gnmn-smaller-entropy} are given in Appendix~\ref{sec:proofs-univ-coding-algorithm}. 

\begin{lem}
\label{lem:local-isomorphism-LP-distance}
  Let $G$ and $G'$ be marked graphs 
  on the vertex set $\{1, \dots, n\}$. For a permutation $\pi \in \mathcal{S}_n$ 
  and an integer $h \ge 0$,
  let $L$ be the number of vertices $1 \leq i \leq n$ such that $(G, i)_h \equiv (G', \pi(i))_h$. Then, we have 
  \begin{equation*}
    \dlp(U(G), U(G')) \leq \max \left \{ \frac{1}{1+h}, 1 - \frac{L}{n} \right \}.
  \end{equation*}
\end{lem}


\begin{lem}
  \label{lem:upper-bound-on-G_n-m_mlogn-n}
  If, for integers $n$ and $0 \leq m \leq \binom{n}{2}$, $\mG_{n,m}$ denotes the
  set of simple unmarked graphs on the vertex set $\{1, \dots, n\}$ with
  precisely $m$ edges,
we have
  \begin{equation*}
    \log |\mG_{n,m}| = \log \left | \binom{\binom{n}{2}}{m} \right | \leq m \log n + ns\left ( \frac{2m}{n} \right ),
  \end{equation*}
where $s(d) := \frac{d}{2} - \frac{d}{2} \log d$ for $d > 0$, and $s(0) := 0$. Moreover, since $s(x) \leq 1/2$ for all $x > 0$, we have, in particular,
\begin{equation*}
  \log |\mG_{n,m}| \leq m \log n + \frac{n}{2}.
\end{equation*}
\end{lem}

\begin{lem}
  \label{lem:log-Gnmn-smaller-entropy}
Assume that a unimodular $\mu \in \mP_u(\mTb_*)$ is given such that $\deg(\mu)
\in (0,\infty)$. Moreover, assume that sequences of  edge and vertex mark count vectors $\vmn$ and
$\vun$ respectively are given such that
\begin{subequations}
  \begin{align}
    &\liminf_{n \rightarrow \infty} \frac{\mn(x,x')}{n} \geq \deg_{x,x'}(\mu), \qquad \forall x \neq x' \in \edgemark; \label{eq:gen-limsup-x-neq}\\
    &\liminf_{n \rightarrow \infty} \frac{\mn(x,x)}{n} \geq \frac{\deg_{x,x}(\mu)}{2}, \qquad \forall x  \in \edgemark; \label{eq:gen-limsup-x-eq}\\
    &\lim_{n \rightarrow \infty} \frac{\un(\theta)}{n} = \vtype_\theta(\mu), \qquad \forall \theta \in \vermark. \label{eq:gen-limsup-theta}
  \end{align}
\end{subequations}
Then, for any sequence $\epsilon_n$ of positive reals converging to zero, 
we have 
  \begin{equation}
    \label{eq:limsup-n-log-Gnm-bchmu}
    \limsup_{n \rightarrow \infty} \frac{ \log |\mG^{(n)}_{\vmn,\vun}(\mu, \epsilon_n)| - \norm{\vmn}_1 \log n}{n} \leq \bch(\mu).
  \end{equation}
\end{lem}


\begin{proof}[Proof of Proposition~\ref{prop:universal-opt-bounded-degree}]
Motivated by our compression scheme, 
we have
  \begin{equation}
\begin{split}
    \label{eq:bounded-degree--lfnGn-upper-bound-1}
    \limsup_{n \rightarrow \infty} \frac{\nat(\tilde{f}_n(\Gn)) - \snorm{\vmn}_1 \log n}{n} &\le \limsup_{n \rightarrow \infty}  \frac{|\mA_{k_n, \Delta_n}| (\log 2 + \log n)}{n} \\
    &\qquad + \frac{\log 2+ \log |W_n| - \snorm{\vmn}_1 \log n}{n} \\
&= \limsup_{n \rightarrow \infty} \frac{\log |W_n| - \snorm{\vmn}_1 \log n}{n},
\end{split}
  \end{equation}
where the last equality employs the assumption $|\mA_{k_n, \Delta_n}| = o(n / \log n)$. 
We now show that
\begin{equation}
\label{eq:Wn-subset-G-epsilon-n-2kn}
  W_n \subseteq \mG^{(n)}_{\vmn,\vun}\left (\mu, \epsilon_n + \frac{1}{1+k_n}\right),
\end{equation}
where $\epsilon_n := \dlp(U(\Gn), \mu)$ and $\vun := \vu_{\Gn}$.
For this, let
$G \in W_n$. By
definition, for all $[G', o'] \in \mA_{k_n, \Delta_n}$, we have $|\psi^{(n)}_{\Gn}([G',
o'])| = |\psi^{(n)}_{G}([G', o'])|$. Hence there exists a permutation $\pi$
on the set of vertices $\{1, \dots, n\}$ such that $(\Gn, i)_{k_n} \equiv (G,
\pi(i))_{k_n}$ for all $1 \leq i \leq n$. Using
Lemma~\ref{lem:local-isomorphism-LP-distance} above with  
$h = k_n$ and $L = n$, 
we have 
\begin{equation*}
  \dlp(U(\Gn), U(G)) \leq \frac{1}{1+k_n}.
\end{equation*}
Consequently,
\begin{equation}
  \label{eq:dlp-UG-mu-epsilon-n-1-kn}
\dlp(U(G), \mu) < \epsilon_n + 1/(1+k_n).
\end{equation}

We claim that for $G \in W_n$ we have $\vm_{G} = \vmn$ and $\vu_G = \vun$.
To see this, note that for $\theta \in \vermark$ we have
\begin{align*}
  u_G(\theta) &= \sum_{i=1}^n \one{\tau_{G}(i) = \theta}\ = \sum_{[G', o'] \in \mA_{k_n, \Delta_n}} 
  \sum_{i \in \psi^{(n)}_{G}([G', o'])} \one{\tau_{G}(i) = \theta}.
\end{align*}
Note that for $i \in \psi^{(n)}_{G}([G', o'])$ we have $\tau_{G}(i) =
\tau_{G'}(o')$. Therefore, 
\begin{align*}
u_G(\theta) &= \sum_{[G', o'] \in \mA_{k_n, \Delta_n}: \tau_{G'}(o') = \theta}  |\psi^{(n)}_G([G',o'])|.
\end{align*}
A similar argument implies 
\begin{equation*}
  \un(\theta) = \sum_{[G', o'] \in \mA_{k_n, \Delta_n}: \tau_{G'}(o') = \theta}  |\psi^{(n)}_{\Gn}([G',o'])|.
\end{equation*}
Hence $\un(\theta) = u_G(\theta)$. Likewise, for $G \in W_n$ and $x \neq x' \in \edgemark$, we
can write,
for $n$ large enough,
\begin{align*}
  m_G(x,x') = \sum_{i=1}^n \deg_G^{x,x'}(i) &= \sum_{[G',o'] \in \mA_{k_n, \Delta_n}} \deg_{G'}^{x,x'}(o') |\psi^{(n)}_{G}([G',o'])| \\
            &= \sum_{[G',o'] \in \mA_{k_n, \Delta_n}} \deg_{G'}^{x,x'}(o') |\psi^{(n)}_{\Gn}([G',o'])| \\
            &= \mn(x,x').
\end{align*}
The proof of $m_G(x,x) = \mn(x,x)$ for $x \in \edgemark$ is similar. 
This, together with \eqref{eq:dlp-UG-mu-epsilon-n-1-kn}, implies that 
$G \in \mG^{(n)}_{\vmn,\vun}(\mu, \epsilon_n + 1/(1+k_n))$ which completes the proof of~\eqref{eq:Wn-subset-G-epsilon-n-2kn}.

Note that, for fixed $t > 0$ and $x, x' \in \edgemark$, the mapping $[G,o]
\mapsto \deg_G^{x,x'}(o) \wedge t$ is bounded and continuous. Therefore, for $x \neq x' \in
\edgemark$, we have 
\begin{align*}
  \frac{\mn(x,x')}{n} &= \int \deg_G^{x,x'}(o) d U(\Gn)([G,o]) \geq  \int (\deg_G^{x,x'}(o)\wedge t) d U(\Gn)([G,o]) \\
  &\xrightarrow{n \rightarrow \infty} \int (\deg_G^{x,x'}(o)\wedge t) d \mu.
\end{align*}
Sending $t$ to infinity, we get 
\begin{equation*}
  \liminf_{n \rightarrow \infty} \frac{\mn(x,x')}{n} \geq \deg_{x,x'}(\mu).
\end{equation*}
Similarly, for $x \in \edgemark$, we have $\liminf_{n \rightarrow \infty}
m_n(x,x) / n \geq \deg_{x,x}(\mu)/2$.
On the other hand, for $\theta \in \vermark$, the mapping $[G,o] \mapsto
\one{\tau_G(o) = \theta}$ is bounded and continuous. This implies that
\begin{equation*}
  \lim_{n \rightarrow \infty} \un(\theta) = \vtype_\theta(\mu).
\end{equation*}
Thus, substituting \eqref{eq:Wn-subset-G-epsilon-n-2kn} into \eqref{eq:bounded-degree--lfnGn-upper-bound-1}, using the fact that $\epsilon_n + 1/(1+k_n) \rightarrow 0$, and using Lemma~\ref{lem:log-Gnmn-smaller-entropy} above, we get 
\begin{align*}
  \limsup_{n \rightarrow \infty} \frac{\nat(\tilde{f}_n(\Gn)) - \snorm{\vmn}_1 \log n}{n} &\leq     \limsup_{n \rightarrow \infty} \frac{\log \left | \mG^{(n)}_{\vmn,\vun}\left (\mu, \epsilon_n + \frac{1}{1+k_n}\right) \right | - \snorm{\vmn}_1 \log n}{n} \\
  &\leq \bch(\mu),
\end{align*}
which completes the proof.
\end{proof}

\subsection{Step 2: The General Compression Scheme}
\label{sec:generel-scheme}

In Section~\ref{sec:coding-step-1-restriction-max-degree} we introduced a
compression scheme which achieves the BC entropy of $\mu$ by focusing on the depth $k_n$ empirical distribution of\
the graph $\Gn$ in the sequence of graphs $\{\Gn\}_{n=1}^\infty$, under the assumption that the
maximum degree of 
$\Gn$ is bounded above by $\Delta_n$ which does not grow too fast, in the sense that
$|\mA_{k_n, \Delta_n}| = o(n / \log n)$.
In principle, we can choose the
design parameter $k_n$, but we have no control over the maximum degree
$\Delta_n$. In order to overcome this and drop the restriction on the compression scheme in
Section~\ref{sec:coding-step-1-restriction-max-degree}, we first choose $k_n$
and $\Delta_n$ and then trim the input graph by removing some edges to make its
maximum degree no more than $\Delta_n$. Then, we encode the resulting trimmed
graph by the compression function in Section~\ref{sec:coding-step-1-restriction-max-degree}.
Finally, we encode the removed edges  separately. 
More precisely, we encode a graph $\Gn \in \mGb_n$ as follows:
\begin{enumerate}
\item Define $\Delta_n :=  \log \log n$.
\item Let $\tGn := {(\Gn)}^{\Delta_n}$ be the trimmed graph obtained by removing the
  edges connected to vertices with degree more than $\Delta_n$. Moreover, define
  \begin{equation*}
    R_n := \{ 1 \leq i \leq n : \deg_{\Gn}(i) > \Delta_n \text{ or } \deg_{\Gn}(j) > \Delta_n \text{ for some } j \sim_{\Gn} i \},
  \end{equation*}
which essentially consists of the endpoints of the removed edges. 
\item Encode the graph $\tGn$ by the compression function introduced in Section~\ref{sec:coding-step-1-restriction-max-degree}, with $k_n = \sqrt{\log \log n}$.
\item Encode $|R_n|$ using 
at most $(1+\lfloor  \log_2 n \rfloor) \log 2$ nats.
\item Encode the set $R_n$ using 
at most $(1+\lfloor  \log_2 \binom{n}{|R_n|} \rfloor) \log 2$ nats.
\item Let $\vmn = \vm_{\Gn}$ and $\vtmn = \vm_{\tGn}$. Note that the edges present in $\Gn$ but not
  in $\tGn$ have both endpoints in the set $R_n$. So we can first encode
  $\mn(x, x') - \tmn(x,x')$ for all $x,x' \in \edgemark$ 
  by $|\edgemark|^2 (1+\lfloor \log_2n^2 \rfloor) \log 2 \leq 2 |\edgemark|^2
  (1+\lfloor  \log_2 n \rfloor) \log 2$ nats
  and then encode these removed  edges 
  using
    \begin{equation*}
   \sum_{x\leq x' \in \edgemark} \left(1+\lfloor  \log_2 \binom{\binom{|R_n|}{2}}{\mn(x,x') - \tmn(x,x')} \rfloor\right) \log 2
  \end{equation*}
  nats by specifying the removed edges of each pair of  marks separately.
\end{enumerate}

Now we show that this general compression scheme achieves the upper BC entropy rate, as was stated in Theorem~\ref{thm:universal-existence-main}. Before this, we need the results of the following lemmas. We postpone the proofs of these lemmas to Appendix~\ref{sec:proofs-univ-coding-algorithm}.

\begin{lem}
  \label{lem:Gn-lwc-trim}
  Assume that $\{\Gn\}_{n=1}^\infty$ is a sequence of marked graphs with local weak limit $\mu \in \mP(\mTb_*)$, where $\Gn$ is on the vertex set $\{1, \dots, n\}$.
  If $\Delta_n$ is a sequence of integers going to infinity as $n\rightarrow \infty$, $\mu$ is also the local weak limit of the trimmed sequence $\{{(\Gn)}^{\Delta_n}\}_{n=1}^\infty$.
\end{lem}

\begin{lem}
  \label{lem:A-k-delta-o(n/logn)}
  If $\Delta_n \leq \log \log n$ and $k_n \leq \sqrt{\log \log n}$, then $|\mA_{k_n, \Delta_n}| = o( n / \log n)$. 
\end{lem}

\begin{lem}
  \label{lem:tn-infty-An-o(n)}
  Assume that $\{ \Gn \}_{n=1}^\infty$ is a sequence of marked graphs with local weak limit $\mu \in \mP(\mTb_*)$, where $\Gn$ is on the vertex set $\{1, \dots, n\}$.
Let $\{\Delta_n\}_{n=1}^\infty$ be a sequence of integers such that $\Delta_n \rightarrow \infty$ and define
  \begin{equation*}
    R_n := \{ 1 \leq i \leq n: \deg_{\Gn}(i) > \Delta_n \text{ or } \deg_{\Gn}(j) > \Delta_n \text{ for some } j \sim_{\Gn} i \}.
  \end{equation*}
Then $|R_n| / n \rightarrow 0$ as $n$ goes to infinity.
\end{lem}

\begin{proof}[Proof of Theorem~\ref{thm:universal-existence-main}]
  Let $\tilde{f}_n$ be the compression function of the scheme in
  Section~\ref{sec:coding-step-1-restriction-max-degree}. 
    We have 
  \begin{align*}
    \nat(f_n(\Gn)) &\le \nat(\tilde{f}_n(\tGn)) + \log n + \log \binom{n}{|R_n|} + 2 |\edgemark|^2 \log n \\
    &\quad +  \sum_{x \leq x' \in \edgemark} \log \binom{\binom{|R_n|}{2}}{\mn(x, x')- \tmn(x,x')} + C \log 2,
\end{align*}
where $C = 2 + 3 |\edgemark|^2$. 
Using the inequality $\binom{r}{s} \leq (re/s)^s$ and
Lemma~\ref{lem:upper-bound-on-G_n-m_mlogn-n} above,  we  have
\begin{align*}
  \nat(f_n(\Gn)) &\leq \nat(\tilde{f}_n(\tGn)) + (1+2 |\edgemark|^2) \log n + |R_n| \log \frac{ne}{|R_n|} \\
  & \qquad + (\snorm{\vmn}_1 - \snorm{\vtmn}_1) \log |R_n| + \frac{|R_n||\edgemark|^2}{2} + C \log 2.
\end{align*}
Using the fact that $|R_n| \leq n$, this gives
\begin{align*}
  \nat(f_n(\Gn)) &\leq \nat(\tilde{f}_n(\tGn)) + (1+2 |\edgemark|^2)\log n + |R_n| \log \frac{ne}{|R_n|} \\
  & \qquad + (\snorm{\vmn}_1 - \snorm{\vtmn}_1) \log n + \frac{|R_n||\edgemark|^2}{2} + C \log 2.
  \end{align*}
Hence,
\begin{equation}
  \label{eq:lfnGn--lftildenGtilden}
\begin{split}
  \limsup_{n \rightarrow \infty} \frac{\nat(f_n(\Gn) - \snorm{\vmn}_1 \log n}{n} &\leq \limsup_{n \rightarrow \infty} \frac{\nat(\tilde{f}_n(\tGn) - \snorm{\vtmn}_1 \log n}{n} \\
&\qquad  +\limsup_{n \rightarrow \infty} \frac{|R_n||\edgemark|^2}{2n} + \limsup_{n \rightarrow \infty} \frac{|R_n|}{n} \log \frac{ne}{|R_n|}.
\end{split}
\end{equation}
Now, we claim that the conditions of
Proposition~\ref{prop:universal-opt-bounded-degree} hold for the sequence $\tGn$
and the parameters $k_n$ and $\Delta_n$ defined above. 
To show this, note that both $k_n$ and $\Delta_n$ go to infinity by definition. 
Lemma~\ref{lem:Gn-lwc-trim}  then implies that $\mu$ is also the local weak limit of the sequence $\tGn$. Moreover, by Lemma~\ref{lem:A-k-delta-o(n/logn)}, $|\mA_{k_n, \Delta_n}| = o(n / \log n)$. On the other hand, the maximum degree in $\tGn$ is at most $\Delta_n$. Therefore, all the conditions of Proposition~\ref{prop:universal-opt-bounded-degree} are satisfied and 
\begin{equation*}
  \limsup_{n \rightarrow \infty} \frac{\nat(\tilde{f}_n(\tGn)) - \snorm{\vtmn}_1 \log n}{n} \leq \bch(\mu).
\end{equation*}

Furthermore, all the other terms in \eqref{eq:lfnGn--lftildenGtilden} go to zero,  since, by Lemma~\ref{lem:tn-infty-An-o(n)}, $|R_n| / n \rightarrow 0$, and the function $\delta \mapsto \delta \log \delta$ goes to zero as $\delta\rightarrow 0$. Therefore,
\begin{equation*}
    \limsup_{n \rightarrow \infty} \frac{\nat(f_n(\Gn)) - \snorm{\vmn}_1 \log n}{n} \leq \limsup_{n \rightarrow \infty} \frac{\nat(\tilde{f}_n(\tGn) - \snorm{\vtmn}_1 \log n}{n}  \leq \bch(\mu),
\end{equation*}
which completes the proof.
\end{proof}

\begin{rem}
From Lemma~\ref{lem:tn-infty-An-o(n)} above, for typical graphs, $|R_n| = o(n)$. Hence, similar to our discussion in Remark~\ref{rem:first-step-local-queries}, we are capable of answering local queries with an error of $o(n)$ without needing to go through the decompression process.
\end{rem}


\section{Conclusion}
\label{sec:conclusion}

We introduced a notion of  stochastic process for graphs, using the language of
local weak convergence. Besides, we discussed a generalized notion of entropy
for such processes. Using this, we formalized the problem of efficiently
compressing graphical data without assuming  prior knowledge of its stochastic
properties. Finally, we proposed a universal compression scheme which 
is asymptotically optimal in the size of the underlying graph, as characterized using 
the discussed notion of entropy,  and is capable of performing local data queries in the compressed form, 
with an error negligible compared to the number of vertices.

\begin{flushleft}
{\bf ACKNOWLEDGMENTS}
\end{flushleft}

Research supported by the NSF Science and Technology Center grant CCF-
0939370: ``Science of Information", the NSF grants ECCS-1343398, CNS-1527846
and CIF-1618145, and by the William and Flora Hewlett Foundation supported Center for Long Term Cybersecurity at Berkeley.


\appendix

\section{Proof of Lemma~\ref{lem:eq-condition-local-weak-convergence}}
\label{sec:proof-equivalence-condition-lwc}

  We first prove that if the condition mentioned in
  Lemma~\ref{lem:eq-condition-local-weak-convergence} is satisfied, then $\mu_n
  \Rightarrow  \mu$. Let $f: \mGb_* \rightarrow \reals$ be a uniformly
  continuous  and bounded function.
Since $f$ is uniformly continuous, for fixed $\epsilon>0$ there exists $\delta
> 0$ such that $|f([G,o]) - f([G',o'])| < \epsilon$ for all $[G, o]$ and
$[G', o']$ such that  $\bar{d}_*([G,o], [G',o'])< \delta$. For this
$\delta$, choose $k$  such that $1/(1+k) < \delta$. 
Note that since $\edgemark$ and $\vermark$ are finite there are countably many locally finite rooted trees with marks in $\edgemark$ and $\vermark$ and depth at most $k$.
Therefore, one can find countably many rooted trees $\{ (T_j, i_j)\}_{j = 1}^\infty$ with depth at most $k$ such that $A^k_{(T_j,i_j)} \cap \mTb_*$ partitions $\mTb_*$.
On the other hand, as $\mu$ is a probability measure with its support being a
subset of $\mTb_*$, one can find finitely many of these $(T_j,i_j)$, which we may index without loss of generality by $1 \leq j \leq m$,
such that $\sum_{j=1}^m \mu(A^k_{(T_j,i_j)}) \geq 1 -  \epsilon$. 
To simplify the notation, we use $A_j$ for $A^k_{(T_j, i_j)}$, $1 \leq j \leq m$.
Note that if $[G,o] \in A_j$, $\bar{d}_*([G,o], [T_j,i_j]) \leq \frac{1}{1+k} < \delta$. Hence, if $\mA$ denotes $\cup_{j=1}^m A_j$, we have 
\begin{equation*}
  \begin{split}
    \left | \int f d \mu - \sum_{j=1}^m f([T_j,i_j]) \mu (A_j) \right | & \leq \sum_{j=1}^m \left | \int_{A_j} f d \mu - f([T_j,i_j]) \mu(A_j)\right| + \norm{f}_\infty \mu(\mA^c) \\
    &\leq \epsilon( 1 + \norm{f}_\infty),
  \end{split}
\end{equation*}
where the last inequality uses the facts that $\mu(\mA^c) \le \epsilon$, $1 / (1+k) < \delta$, and $|f([G,o]) - f([T_j,i_j])| < \epsilon$ for $[G, o] \in A_j$. Similarly, we have 
\begin{equation*}
\begin{split}
    \left | \int f d \mu_n - \sum_{j=1}^m f([T_j,i_j]) \mu(A_j) \right | &\leq \left | \int f d \mu_n - \sum_{j=1}^m f([T_j,i_j]) \mu_n(A_j) \right | + \sum_{j=1}^m |f([T_j,i_j])| | \mu_n(A_j) - \mu(A_j) | \\
    &\leq \norm{f}_\infty \left ( 1 - \sum_{j=1}^m \mu_n(A_j) \right ) + \epsilon +\norm{f}_\infty\sum_{j=1}^m | \mu_n(A_j) - \mu(A_j)|.
\end{split}
\end{equation*}
Combining the two preceding inequalities, we have 
\begin{equation*}
  \left | \int f d \mu_n - \int f d \mu \right | \leq \norm{f}_\infty \left ( 1 - \sum_{j=1}^m \mu_n(A_j) \right )  + \norm{f}_\infty  \sum_{j=1}^m |\mu_n(A_j) - \mu(A_j) | + \epsilon(2 + \norm{f}_\infty).
\end{equation*}
Now, as $n$ goes to infinity, $\mu_n(A_j) \rightarrow \mu(A_j)$ by assumption and also $1 - \sum_{j=1}^m \mu_n(A_j)  \rightarrow \mu(\mA^c) \le \epsilon$. Thus, 
\begin{equation*}
  \limsup_{n \rightarrow \infty} \left | \int f d\mu_n - \int f d \mu \right | \leq 2\epsilon(1 + \norm{f}_\infty).
\end{equation*}
Since $\norm{f}_\infty < \infty$ and $\epsilon > 0$ is arbitrary, $\int f d \mu_n \rightarrow \int f d \mu$, whereby $\mu_n \Rightarrow \mu$.

For the converse, 
fix an integer $h \ge 0$ 
and a rooted marked tree $(T,i)$ with
depth at most $h$.   
Since $\oneu{A^h_{(T,i)}}([G, o]) = \oneu{A^h_{(T,i)}}([G',o'])$ when $\bar{d}_*([G,o], [G',o']) < 1 / (1+h)$, we see that $\oneu{A^h_{(T,i)}}$ is a bounded continuous function. This immediately implies that
\begin{equation*}
  \mu_n(A^h_{(T, i)}) = \int \oneu{A^h_{(T, i)}} d \mu_n \rightarrow \int \oneu{A^h_{(T, i)}} d \mu =   \mu(A^h_{(T, i)}),
\end{equation*}
which completes the proof. 


\section{Proofs for Section~\ref{sec:univ-coding-algor}}
\label{sec:proofs-univ-coding-algorithm}

\begin{proof}[Proof of Lemma~\ref{lem:local-isomorphism-LP-distance}]
  Let $\mA$ be the set of $1 \leq i \leq n$ such that 
  $[G, i]_h = [G', \pi(i)]_h$.
  Then, for any Borel set $B \subset \mGb_*$,  we have 
  \begin{align*}
    U(G)(B) &= \frac{1}{n} \sum_{i=1}^n \one{[G, i] \in B} \\
            &\leq \frac{1}{n} \sum_{i \in \mA} \one{[G, i] \in B} + 1 - \frac{L}{n}.
  \end{align*}
Note that if for some $i \in \mA$ we have $[G, i] \in B$ then,  since 
$(G, i)_h  \equiv (G', \pi(i))_h$, 
we have $d_*([G, i], [G', \pi(i)]) \leq \frac{1}{1+h}$.  This means
that, for such $i$, $[G', \pi(i)] \in B^{\delta+1/(1+h)}$ for arbitrary $\delta> 0$. Continuing
the chain of inequalities, we have
\begin{align*}
  U(G)(B) &\leq \frac{1}{n} \sum_{i\in \mA} \one{[G', \pi(i)] \in B^{\delta+1/(1+h)}} + 1 -\frac{L}{n} \\
    & \leq \frac{1}{n} \sum_{i=1}^n \one{[G', i] \in B^{\delta+1/(1+h)} } + 1 - \frac{L}{n}\\
    & = U(G')(B^{\delta+1/(1+h)}) + 1 - \frac{L}{n}.
  \end{align*}
Changing the order of $G$ and $G'$, we have 
\begin{equation*}
  \dlp(U(G), U(G')) \leq \max \left \{ \frac{1}{1+h} + \delta, 1 - \frac{L}{n} \right \}.
\end{equation*}
We get the desired result by sending $\delta$ to zero.
\end{proof}

\begin{proof}[Proof of Lemma~\ref{lem:upper-bound-on-G_n-m_mlogn-n}]
  Using the classical upper bound $\binom{r}{s} \leq (re/s)^s$, we have 
  \begin{equation*}
    \log \left | \binom{\binom{n}{2}}{m} \right | \leq m \log \frac{n^2 e}{2m} = m \log n + m \log \frac{ne}{2m} = m \log n + n s(2m/n),
  \end{equation*}
which completes the proof of the first statement. Also, it is easy to see that $s(x)$ is increasing for $x<1$, decreasing for $x> 1$ and attains its maximum value $1/2$ at $x=1$. Therefore, $s(x) \leq 1/2$. This completes the proof of the second statement. 
\end{proof}

Next, we prove Lemma~\ref{lem:log-Gnmn-smaller-entropy}.  Before that, we state
and  prove
the following lemmas which will be useful in our proof. 
For a marked graph $G$ on a finite or countably infinite vertex set, let
$\mathsf{UM}(G)$ denote the unmarked graph which has the same set of vertices and
edges as in $G$, but is obtained from $G$ by removing all the vertex and edge
marks. Given a probability distribution $\mu \in \mP(\mG_*)$  on the space of
isomorphism classes of rooted unmarked graphs, for $\epsilon>0$ and integers $n$ and $m$, let
$\mG_{n,m}(\mu, \epsilon)$ denote the set of unmarked graphs $G$ on the vertex
set $\{1, \dots, n\}$ with $m$ edges such that $\dlp(U(G), \mu) < \epsilon$.

\begin{lem}
  \label{lem:Tx-continuous}
For $[G,o]$ and $[G',o']$ in $\mGb_*$, we have $d_*([\mathsf{UM}(G), o], [\mathsf{UM}(G'), o']) \leq \bar{d}_*([G, o], [G', o'])$.
\end{lem}

\begin{proof}
  By definition, for $\epsilon>0$, the condition $\bar{d}_*([G, o], [G', o']) < \epsilon$ means that for
  some $k$ with $1/(1+k) < \epsilon$, we have $[G, o]_k \equiv [G', o']_k$.
  This implies $[\mathsf{UM}(G), o]_k \equiv [\mathsf{UM}(G'), o']_k$, which in particular means
  that  $d_*([\mathsf{UM}(G), o], [\mathsf{UM}(G'), o'])\leq 1/(1+k) < \epsilon$. 
\end{proof}

\begin{lem}
  \label{lem:Gnmn-Tx}
  Assume $\mu \in \mP(\mGb_*)$ is given. Let $\tilde{\mu} \in \mP(\mG_*)$ be the
  law of $[\mathsf{UM}(G), o]$ when $[G,o]$ has law $\mu$. Then, given an
  integer $n$, edge and vertex mark count vectors $\vmn$ and $\vun$ respectively, and
  $\epsilon>0$, for all $G \in \mGnmnun(\mu, \epsilon)$, we have $\mathsf{UM}(G)
  \in \mG_{n, m_n}(\tilde{\mu}, \epsilon)$ where $m_n:= \snorm{\vmn}_1$.
\end{lem}


\begin{proof}
Fix $G \in \mGnmnun(\mu, \epsilon)$.
  Note that $\mathsf{UM}(G)$ has $m_n$ edges, and we only need to show that
  $\dlp(U(\mathsf{UM}(G)), \tilde{\mu}) < \epsilon$. 
Let $\delta:= \dlp(U(G), \mu)$. This means that for all $\delta' > \delta$, and
for all Borel sets $A$ in $\mGb_*$, we have $(U(G))(A) \leq \mu(A^{\delta'}) +
\delta'$ and $\mu(A) \leq (U(G))(A^{\delta'}) + \delta'$, where, $A^{\delta'}$
denotes the $\delta'$--extension of $A$. 
Define $T: \mGb_* \rightarrow \mG_*$ that maps $[G,o] \in
\mGb_*$ to $[\mathsf{UM}(G), o] \in \mG_*$. 
Lemma~\ref{lem:Tx-continuous} above implies that $T$ is continuous and in fact 1--Lipschitz. It is easy to see that
$U(\mathsf{UM}(G))$ is the pushforward of $U(G)$ under the mapping $G$. Also,
$\tilde{\mu}$ is the pushforward of $\mu$ under $T$. Using the fact that $T$ is
1--Lipschitz, it is easy to see that for any Borel set $B$ in $\mG_*$, and any
$\zeta > 0$,  we have $(T^{-1}(B))^{\zeta} \subset T^{-1}(B^{\zeta})$. Putting
these together, 
  for
$\delta' > \delta$ and a Borel set $B$ in $\mG_*$, we have 
\begin{equation*}
    U(\mathsf{UM}(G))(B) = U(G)(T^{-1}(B)) 
    \leq \mu((T^{-1}(B))^{\delta'}) + \delta'
    \leq \mu(T^{-1}(B^{\delta'})) + \delta'
    = \tilde{\mu}(B^{\delta'}) + \delta'.
\end{equation*}
Similarly,
\begin{align*}
  \tilde{\mu}(B) = \mu(T^{-1}(B)) \leq (U(G))((T^{-1}(B))^{\delta'}) + \delta' &\leq (U(G))(T^{-1}(B^{\delta'})) + \delta' \\
  &= (U(\mathsf{UM}(G)))(B^{\delta'}) + \delta'.
\end{align*}
Since this holds for any $\delta' > \delta$ and any Borel set $B$ in $\mG_*$, we
have $\dlp(U(\mathsf{UM}(G)), \tilde{\mu}) \leq \delta = \dlp(U(G), \mu) <
\epsilon$. Consequently, we have $\mathsf{UM}(G) \in \mG_{n, m_n}(\tilde{\mu},
\epsilon)$ and the proof is complete. 
\end{proof}

Now, we are ready to prove Lemma~\ref{lem:log-Gnmn-smaller-entropy}.

\begin{proof}[Proof of Lemma~\ref{lem:log-Gnmn-smaller-entropy}]
To simplify the notation, for  $\epsilon > 0$ define
\begin{align*}
  a_n(\epsilon) := \frac{\log|\mGnmnun(\mu, \epsilon)| - \snorm{\vmn} \log n}{n}.
\end{align*}
Note that there exists a subsequence $\{n_k\}$ such that $\limsup_{n \rightarrow
\infty}
a_n(\epsilon_n) = \lim_{k \rightarrow \infty}
a_{n_k}(\epsilon_{n_k})$.
Moreover, there is a further subsequence $n_{k_r}$ such that for all $x,x' \in
\edgemark$, there exists $\bar{d}_{x,x'} \in [0,\infty]$ where
\begin{subequations}
\begin{align}
  \frac{m^{(n_{k_r})}(x,x')}{n_{k_r}} &\rightarrow \bar{d}_{x,x'},  \qquad x \neq x'; \label{eq:gen-limsup-bard-neq}\\
  \frac{2m^{(n_{k_r})}(x,x)}{n_{k_r}} & \rightarrow \bar{d}_{x,x}. 
  \label{eq:gen-limsup-bard-eq}
\end{align}
\end{subequations}
Observe that since $a_{n_k}(\epsilon_{n_k})$ is convergent, it suffices that we focus on
the subsequence $\{n_{k_r}\}$ and show that $\lim
a_{n_{k_r}}(\epsilon_{n_{k_r}}) \leq \bch(\mu)$.
Note that due to conditions~\eqref{eq:gen-limsup-x-neq} and \eqref{eq:gen-limsup-x-eq}, we have $\bar{d}_{x,x'} \geq
\deg_{x,x'}(\mu)$ for all $x, x' \in \edgemark$. Therefore, there are  two possible
cases: either $\bar{d}_{x,x'} = \deg_{x,x'}(\mu)$ for all $x, x' \in \edgemark$,
or there exist $x, x' \in \edgemark $ such that
$\bar{d}_{x,x'} > \deg_{x,x'}(\mu)$. To simplify the notation, without loss of
generality,  we may assume that
the subsequence $n_{k_r}$ is the whole sequence, i.e.\ $a_n(\epsilon_n)$ is
convergent, and~\eqref{eq:gen-limsup-bard-neq} and~\eqref{eq:gen-limsup-bard-eq}
hold for the whole sequence. 

\underline{Case 1:} $\bar{d}_{x,x'} = \deg_{x,x'}(\mu)$ for all $x, x' \in
\edgemark$. 
We define edge and vertex mark count
vectors $\vtmn = (\tmn(x,x') : x,x' \in \edgemark)$ and $\vtun = (\tun(\theta):
\theta \in \vermark)$ as follows. For $x, x' \in \edgemark$, define $\tmn(x,x')$
 to be $\mn(x,x')$  if $\deg_{x,x'}(\mu) > 0$ and $0$ otherwise. Also, fix some
 $\theta_0 \in \vermark$ such that $\vtype_{\theta_0}(\mu) > 0$. For $\theta
 \in \vermark$, define
 \begin{equation*}
   \tun(\theta) =
   \begin{cases}
     0, & \text{ if } \vtype_\theta(\mu) = 0, \\
     \un(\theta), & \text{ if } \vtype_\theta(\mu) > 0 \text{ and } \theta \neq \theta_0, \\
     \un(\theta_0) + \sum_{\theta': \vtype_{\theta'}(\mu) = 0 } \un(\theta'), & \text{ if } \theta = \theta_0.
   \end{cases}
 \end{equation*}
Note that, by construction and from~\eqref{eq:gen-limsup-bard-neq} and
\eqref{eq:gen-limsup-bard-eq}, the sequences $\vtmn$ and $\vtun$ are adapted to $(\vdeg(\mu),
\vvtype(\mu))$. Also, $\mn(x,x') \geq \tmn(x,x')$ for all $n$ and all $x,x' \in \edgemark$.

Now, fix $\epsilon>0$, and pick an integer $h$ such that $1/(1+h) < \epsilon$. Define $B$ to the set of $[G,o] \in \mGb_*$ such that either for some $x,x' \in
\edgemark$ with $\deg_{x,x'}(\mu) = 0$ there exists an edge in $[G,o]_h$ with
pair of marks $x,x'$, or, for some $\theta \in \vermark$ with $\vtype_\theta(\mu)
= 0$, there exists a vertex in $[G,o]_h$ with mark $\theta$. Then, from
Lemma~\ref{lem:everything-root}, we have $\mu(B) = 0$. On the other hand, for
$n$ large enough such that $\epsilon_n < 1/(1+h)$, we have $B^{\epsilon_n} = B$.
Hence, for large enough $n$ and $G \in \mGnmnun(\mu, \epsilon_n)$, we have $U(G)(B)
\leq \epsilon_n$. For such $G$, we construct a marked graph $\tG$ which is
obtained from $G$ by removing all edges which have their pair of marks $x,x'$ with
$\deg_{x,x'}(\mu) = 0$. Moreover, if a vertex $v$ in $G$ has mark $\theta$ with
$\vtype_\theta(\mu) = 0$, we change its mark to $\theta_0$ in $\tG$, with
the $\theta_0$ defined above. Note that
$\tG \in \mGn_{\vtmn, \vtun}$. Furthermore, since $U(G)(B) \leq \epsilon_n$, the
number of vertices $v$ in $G$ such that $(G,v)_h \equiv (\tG,v)_h$ is at least
$n(1 - \epsilon_n)$. Consequently, using
Lemma~\ref{lem:local-isomorphism-LP-distance}, when $n$ is large enough so that
$\epsilon_n < \epsilon$, we have $\dlp(G, \tG) \leq \max
\left\{ 1/(1+h), \epsilon_n \right\} < \epsilon$.
This means that $\tG \in \mGn_{\vtmn, \vtun}(\mu, \epsilon_n + \epsilon) \subset
\mGn_{\vtmn, \vtun}(\mu, 2\epsilon)$.

Motivated by this discussion, for $n$ large enough, we have 
\begin{equation}
  \label{eq:lem-limsup-mGnmnun-mGntmntun-bound}
  \begin{aligned}
    |\mGnmnun(\mu, \epsilon_n) | &\leq |\mGn_{\vtmn, \vtun}(\mu, 2\epsilon) | \times \prod_{\stackrel{x \leq x' \in \edgemark}{\deg_{x,x'}(\mu) = 0}} |\mG_{n, \mn(x,x') - \tmn(x,x')} | \times 2^{\mn(x,x') - \tmn(x,x')} \\
&\qquad \times \prod_{\theta \in \vermark} \binom{n}{|\un(\theta) - \tun(\theta)|}.
  \end{aligned}
\end{equation}
Here we have assumed that, since $\edgemark$ is finite, it is an ordered set.
For $x \leq x' \in \edgemark$ with $\deg_{x,x'}(\mu) = 0$, using
Lemma~\ref{lem:upper-bound-on-G_n-m_mlogn-n}, we have
\begin{align*}
  \log\left(  |\mG_{n, \mn(x,x') - \tmn(x,x')} | \times 2^{\mn(x,x') - \tmn(x,x')} \right) &\leq (\mn(x,x') - \tmn(x,x')) \log n \\
                                                                                           &\quad+ ns\left( \frac{2(\mn(x,x') - \tmn(x,x'))}{n} \right) \\
                                                                                            &\quad+ (\mn(x,x') - \tmn(x,x')) \log 2.
\end{align*}
Note that, for all $x,x' \in \edgemark$, $\frac{1}{n} (\mn(x,x') - \tmn(x,x'))
\rightarrow 0$. Also, for all $\theta \in \vermark$, $\frac{1}{n} |\un(\theta) -
\tun(\theta)| \rightarrow 0$. Additionally, $s(y)
\rightarrow 0$ as $y \rightarrow 0$. Using these 
in~\eqref{eq:lem-limsup-mGnmnun-mGntmntun-bound} and simplifying, we get
\begin{align*}
  \limsup_{n\rightarrow \infty} \frac{\log|\mGnmnun(\mu, \epsilon_n)| - \snorm{\vmn}_1 \log n}{n} &\leq \limsup_{n \rightarrow \infty}\frac{\log|\mGn_{\vtmn, \vtun}(\mu, 2\epsilon)| - \snorm{\vtmn}_1 \log n}{n} \\
  &\leq \bchover_{\vdeg(\mu), \vvtype(\mu)}(\mu, 2\epsilon)|_{\vtmn, \vtun},
\end{align*}
where the last inequality employs the fact that, by construction, $\vtmn$ and
$\vtun$ are adapted to $(\vdeg(\mu), \vvtype(\mu))$. The above inequality
holds for all $\epsilon> 0$. Therefore, from Theorem~\ref{thm:bch-properties}, as $\epsilon \rightarrow 0$, the right hand side converges to 
$\bch(\mu)$. This completes the proof for this case.

\underline{Case 2:} $\bar{d}_{x,x'} > \deg_{x,x'}(\mu)$ for some $x,x' \in
\edgemark$.  Let $\bar{d} := \sum_{x,x' \in \edgemark} \bar{d}_{x,x'}$. Note that
$\bar{d} > \deg(\mu)$. First,
assume that $\bar{d} = \infty$. Observe that
\begin{equation*}
  |\mGnmnun(\mu, \epsilon)| \leq |\mGnmnun| \leq |\vermark|^n \prod_{x \leq x' \in \edgemark} |\mG_{n, \mn(x,x')}| 2^{\mn(x,x')}.
\end{equation*}
Using Lemma~\ref{lem:upper-bound-on-G_n-m_mlogn-n},
\begin{equation}
  \label{eq:gen-limsup-infty-case-bound-1}
\begin{aligned}
  |\mGnmnun(\mu, \epsilon)| &\leq n \log |\vermark| + \snorm{\vmn}_1 \log n + n \sum_{x \leq x' \in \edgemark} s\left( \frac{2\mn(x,x')}{n} \right) + \mn(x,x') \log 2 \\
  &= n \log |\vermark| + \snorm{\vmn}_1 \log n + 2n \sum_{x \leq x' \in \edgemark} s\left( \frac{\mn(x,x')}{n} \right).
\end{aligned}
\end{equation}
Since we have assumed $\bar{d}= \infty$, there exist $\bar{x} \leq \bar{x}' \in
\edgemark$ such that $\bar{d}_{\bar{x}, \bar{x}'} = \infty$.
Therefore,~\eqref{eq:gen-limsup-bard-neq} and \eqref{eq:gen-limsup-bard-eq}
imply that $\mn(\bar{x}, \bar{x}') / n \rightarrow \infty$. On the other hand,
$s(y) \rightarrow -\infty$ as $y \rightarrow \infty$. Using these in
\eqref{eq:gen-limsup-infty-case-bound-1}, we get $\limsup_{n \rightarrow \infty}
a_n(\epsilon_n) = -\infty$ which completes the proof. Therefore, it remains to
consider the case $\bar{d} < \infty$.

 Let $\tilde{\mu} \in \mP(\mTb_*)$ be the law of $[\mathsf{UM}(T),o]$ when
$[T,o]$ has law $\mu$, and let $m_n := \snorm{\vmn}_1$. From
Lemma~\ref{lem:Gnmn-Tx}, if $G \in \mGnmnun(\mu, \epsilon_n)$, we have $\mathsf{UM}(G)
\in \mG_{n, m_n}(\tilde{\mu}, \epsilon_n)$.
Moreover, by finding an upper bound on the number of possible ways to mark
vertices and edges for an unmarked graph in $\mG_{n,m_n}$, we have 
\begin{equation}
  \label{eq:limsup-gen-bound-2}
  |\mGnmnun(\mu, \epsilon_n)| \leq |\mG_{n, m_n}(\tilde{\mu}, \epsilon_n)| \times |\vermark|^n \times \frac{m_n!}{\prod_{x \leq x'} \mn(x,x')!} \times 2^{m_n}.
\end{equation}
Note that $m_n / n \rightarrow \bar{d}/2 < \infty$, and $\mn(x,x') / n$
converges to $\bar{d}_{x,x'}/2$ when $x = x'$, and $\bar{d}_{x,x'}$ when $x
\neq x'$. Hence, 
\begin{align*}
  \lim_{n \rightarrow \infty} \frac{1}{n} \log \left( |\vermark|^n \times \frac{m_n!}{\prod_{x \leq x'} \mn(x,x')!} \times 2^{m_n}  \right) &= \log |\vermark| + \sum_{x < x' \in \edgemark} \bar{d}_{x,x'} \log \frac{\bar{d}}{\bar{d}_{x,x'}} \\
  &\quad +\sum_{x \in \edgemark} \frac{\bar{d}_{x,x}}{2} \log \frac{2\bar{d}}{\bar{d}_{x,x}}  =: \alpha.
\end{align*}
Note that, as $\bar{d} < \infty$, $\alpha$ is a bounded real number. Also, since
$\epsilon_n \rightarrow 0$, for each $\epsilon>0$ fixed we have $\epsilon_n< \epsilon$  for
$n$ large enough. Putting these in~\eqref{eq:limsup-gen-bound-2}, we get
\begin{equation}
  \label{eq:gen-limsup-bound-3}
  \limsup_{n \rightarrow \infty} a_n(\epsilon_n) \leq \alpha + \limsup_{n \rightarrow \infty} \frac{\log|\mG_{n, m_n}(\tilde{\mu}, \epsilon)| - m_n \log n}{n}.
\end{equation}
Note that $\deg(\tilde{\mu})  = \deg(\mu)$. Moreover, $m_n / n \rightarrow \bar{d}
> \deg(\mu) = \deg(\tilde{\mu}) > 0$. Furthermore, our notion of marked BC entropy
reduces to the unmarked BC entropy of \cite{bordenave2014large} when
$\vermark$ and $\edgemark$ have cardinality one. Therefore, since $\bar{d} \neq
\deg(\tilde{\mu})$, from part~3 of Theorem~\ref{thm:badcases}, (or equivalently
from part 3 of Theorem~1.2 in \cite{bordenave2014large}), the right hand side
of~\eqref{eq:gen-limsup-bound-3} goes to $-\infty$ as $\epsilon\rightarrow 0$.
Therefore, $\limsup_{n\rightarrow \infty} a_n(\epsilon_n) = -\infty$, which
completes the proof.
\end{proof}

\begin{proof}[Proof of Lemma~\ref{lem:Gn-lwc-trim}]
Let $\mu_n$ and $\tilde{\mu}_n$ denote $U(\Gn)$ and $U((\Gn)^{\Delta_n})$ respectively.
  For an integer $k \ge 0$ 
  and a marked rooted tree $(T, i)$ with depth at most $k$, define
  \begin{equation*}
    A^k_{(T, i)} := \{ [G, o] \in \mGb_*: (G, o)_k \equiv (T, i) \},
  \end{equation*}
  as in \eqref{eq:Ah}.
  From Lemma~\ref{lem:eq-condition-local-weak-convergence} in Section~\ref{sec:local-weak-conv}, in order to show $\tilde{\mu}_n \Rightarrow \mu$, it suffices to show that $\tilde{\mu}_n(A^k_{(T, i)}) \rightarrow \mu(A^k_{(T, i)})$ for all such $k$ and $(T,i)$.  We will now do this.
  Fix some integer $k \ge 0$ 
    throughout the rest of the discussion.
  For an integer $\Delta$, define
  \begin{equation*}
    B^\Delta 
    = \{ [G, o] \in \mGb_*: \deg_{G}(j) > \Delta \text{ for some } j \text{ with distance at most } k+1 \text{ from } o \}.
  \end{equation*}
With this, we have 
\begin{equation}
  \label{eq:lim-Delta-infty-mu-BkDelta}
  \begin{split}
    \lim_{\Delta \rightarrow \infty} 
    \mu((B^\Delta)^c) 
    &= \mu\left ( \bigcup_{\Delta=1}^\infty 
    (B^\Delta)^c \right ) \\
    &= \mu(\deg_{G}(j) < \infty \text{ for all $j$ with distance at most $k+1$ from $o$} ) \\
    &= 1,
  \end{split}
\end{equation}
where the last equality comes from the fact that all graphs in $\mGb_*$ are locally finite. Next, define 
\begin{equation*}
\begin{split}
  C_n &:= \{ i \in V(\Gn): \deg_{\Gn}(j) \leq \Delta_n \text{ for all nodes $j$ in $\Gn$ with distance at most $k+1$ from $i$} \} \\
  &= \{ i \in V(\Gn): [\Gn, i] \in 
  (B^{\Delta_n})^c \}.
\end{split}
\end{equation*}
Now, since $V(\Gn) = \{1, \dots, n\}$, we have 
\begin{equation*}
  \begin{split}
    \tilde{\mu}_n(A_{(T, i)}) &= \frac{1}{n} \sum_{j =1}^n \one{({(\Gn)}^{\Delta_n} , j)_k \equiv (T, i)} \\
    &= \frac{1}{n} \sum_{j \in C_n} \one{({(\Gn)}^{\Delta_n}, j)_k \equiv (T, i)} + \frac{1}{n} \sum_{j \in C_n^c} \one{({(\Gn)}^{\Delta_n}, j)_k \equiv (T, i)} \\
    &= \frac{1}{n} \sum_{j \in C_n} \one{(\Gn, j)_k \equiv (T, i)} + \frac{1}{n} \sum_{j \in C_n^c} \one{({(\Gn)}^{\Delta_n}, j)_k \equiv (T, i)}.
  \end{split}
\end{equation*}
Comparing this to 
\begin{equation*}
  \mu_n(A_{(T, i)}) = \frac{1}{n} \sum_{j=1}^n \one{(\Gn, j)_k \equiv (T, i)},
\end{equation*}
we realize that
\begin{equation*}
  |\tilde{\mu}_n(A_{(T, i)}) - \mu_n(A_{(T, i)}) | \leq \frac{1}{n} |C_n^c| = 
  \mu_n(B^{\Delta_n}).
\end{equation*}
Now, fix an integer  $\Delta>0$. Since $\Delta_n \rightarrow \infty$, we have $\Delta < \Delta_n$ for $n$ large enough. Moreover, as 
$B^\Delta$ is closed,
\begin{equation*}
  \limsup_{n \rightarrow \infty}   |\tilde{\mu}_n(A_{(T, i)}) - \mu_n(A_{(T, i)}) |  \leq \limsup_{n \rightarrow \infty} 
  \mu_n(B^\Delta) \leq \mu(B^\Delta).
\end{equation*}
This is true for all $\Delta>0$; therefore, sending $\Delta$ to infinity and
using~\eqref{eq:lim-Delta-infty-mu-BkDelta}, we have $ |\tilde{\mu}_n(A_{(T,
  i)}) - \mu_n(A_{(T, i)}) | \rightarrow 0$. On the other hand, we have assumed
that  $\mu_n \Rightarrow \mu$. Therefore,
Lemma~\ref{lem:eq-condition-local-weak-convergence} implies that $\mu_n(A_{(T,
  i)}) \rightarrow \mu(A_{(T, i)})$. This means that $\tilde{\mu}_n(A_{(T, i)})
\rightarrow \mu(A_{(T, i)})$. Since this is true for all $k$ and $(T, i)$,
Lemma~\ref{lem:eq-condition-local-weak-convergence} implies that $\tilde{\mu}_n
\Rightarrow \mu$, which completes  the proof.
\end{proof}

\begin{proof}[Proof of Lemma~\ref{lem:A-k-delta-o(n/logn)}]
  In order to count $|\mA_{k_n, \Delta_n}|$, note that, for $\Delta_n \ge 2$, a rooted graph of depth at most $k_n$ and maximum degree at most $\Delta_n$ has at most 
  \begin{equation*}
    1 + \Delta_n + \Delta_n^2 + \dots + \Delta_n^{k_n} \leq \Delta_n^{k_n+1},
  \end{equation*}
many vertices, each of which has $|\vermark|$ many choices for the vertex mark.
On the other hand, such a graph can have at most $\Delta_n^{2(k_n+1)}$ many
edges, each of which can be present or not, and, if present, has $|\edgemark|^2$
many choices for the edge  mark. Consequently, 
\begin{equation*}
  |\mA_{k_n, \Delta_n}| \leq (1+|\edgemark|^2)^{\Delta_n^{2(k_n +1)}}|\vermark|^{\Delta_n^{k_n+1}}. 
\end{equation*}
Therefore,
\begin{align*}
  \log |\mA_{k_n, \Delta_n}| &\leq \Delta_n^{2(1+k_n)} \log (1+|\edgemark|^2) + \Delta_n^{1+k_n} \log |\vermark| \\
                             &\leq \Delta_n^{2(1+k_n)} \log (|\vermark|(1+|\edgemark|^2))  \leq \Delta_n^{4k_n} \log (|\vermark|(1+|\edgemark|^2)),
\end{align*}
where the last inequality holds for $n$ large enough such that $k_n \geq 1$.
Note that in order to show $|\mA_{k_n, \Delta}| = o(n / \log n)$, it suffices to
show that $\log|\mA_{k_n, \Delta_n}| - \log(n/\log n) \rightarrow -\infty$.
Motivated by the above inequality, we observe that this is satisfied if $\Delta_n^{4k_n} =
O(\sqrt{\log n})$. Suppose now that $\Delta_n \le \log \log n$ and 
$k_n \le \sqrt{ \log \log n }$. For $n$ large enough, we have 
\begin{equation*}
  \log (\Delta_n^{4k_n}) \le 4 \sqrt{\log \log n} \log \log \log n \leq \frac{1}{2}\sqrt{\log \log n} \sqrt{\log \log n} = \frac{1}{2} \log \log n.
\end{equation*}
This means that for $n$ large enough we have $\Delta_n^{4k_n} \leq \sqrt{\log n}$.
This completes the proof.
\end{proof}

\begin{proof}[Proof of Lemma~\ref{lem:tn-infty-An-o(n)}]
  For $\Delta > 0$, define $B_\Delta \subset \mGb_*$ as 
  \begin{equation*}
    B_\Delta := \{ [G, o] \in \mGb_*: \deg_{G}(o) \le \Delta \text{ and } \deg_{G}(i) \le \Delta \text{ for all } i \sim_{G} o \}.
  \end{equation*}
Since all graphs in $\mGb_*$ are locally finite, we have $\mu(B_\Delta) \rightarrow 1$ as $\Delta \rightarrow \infty$. On the other hand, 
\begin{equation*}
  \frac{|R_n|}{n} = U(\Gn)(B_{\Delta_n}^c).
\end{equation*}
Since $\Delta_n \rightarrow \infty$, 
we have $B_\Delta\subseteq B_{\Delta_n}$ for $n$ large enough,
for any value of $\Delta$.  Moreover, $B_\Delta$ is both open and closed. Therefore,
\begin{equation*}
  \frac{|R_n|}{n} = U(\Gn)(B_{\Delta_n}^c) \leq U(\Gn)(B_\Delta^c) \rightarrow \mu(B_\Delta^c).
\end{equation*}
But this is true for all $\Delta$, and $\mu(B_\Delta) \rightarrow 1$ as $\Delta \rightarrow \infty$. Consequently, $|R_n|/n \rightarrow 0$, and the proof is complete. 
\end{proof}





\end{document}